\documentclass[12pt]{article}

\usepackage[english]{babel}	
\usepackage[utf8x]{inputenc}	
\usepackage{enumerate}
\usepackage{typearea}
\usepackage{pdfpages}
\usepackage[toc]{appendix}
\usepackage{subcaption}
\usepackage{verbatim}
\usepackage{comment}
\usepackage{mathtools}
\usepackage[hidelinks]{hyperref}
\mathtoolsset{showonlyrefs,showmanualtags}

\usepackage{enumitem}
\usepackage{amsfonts, amsmath, amssymb, amsthm, dsfont}
\usepackage{natbib}
\allowdisplaybreaks
\usepackage{booktabs}
\usepackage[noblocks]{authblk}
\usepackage{ulem}

\theoremstyle{definition}
\newtheorem{definition}{Definition}[section]

\newtheorem*{assumption*}{Assumption}

\newtheorem*{condition*}{Condition}
\newtheorem{example}[definition]{Example}
\theoremstyle{plain}
\newtheorem{theorem}[definition]{Theorem}
\newtheorem{proposition}[definition]{Proposition}
\newtheorem{lemma}[definition]{Lemma}
\newtheorem{cor}[definition]{Corollary}
\theoremstyle{remark}
\newtheorem{remark}{Remark}

\newcommand{\Hyp}{\mathcal{H}}
\newcommand{\N}{\mathbb{N}}

\newcommand{\R}{\mathbb{R}}

\newcommand{\E}{\mathbb{E}}

\newcommand{\F}{\mathcal{F}}

\newcommand{\asconv}{\xrightarrow{a.s.}}

\newcommand{\Cov}{\operatorname{Cov}}

\newcommand{\deq}{\overset{d}{=}}
\newcommand{\wconv}{\Rightarrow}

\newcommand{\B}{B_{\Psi_2, \infty}^{1/2}}

\title{Besov-Orlicz moduli of Brownian motion and polygonal partial sum processes} 
\author{Fabian Mies\thanks{Correspondence: f.mies@tudelft.nl\\ This publication is part of the project VI.Veni.242.365 of the NWO Talent Programme, which is financed by the Dutch Research Council (NWO) under the grant https://doi.org/10.61686/AGAFX90293.}\\ Delft University of Technology\\ Delft, The Netherlands
}

\begin{document}

\maketitle

\begin{abstract}
\noindent
The sample paths of Brownian motion are known to admit the exact Besov-type smoothness exponent 1/2 when measured in the sub-Gaussian Orlicz norm. 
We extend these regularity results by deriving the exact limit of the sub-Gaussian Orlicz modulus for Brownian motion in Banach spaces, and we provide a rate of convergence towards this limiting value. 
The central technique is a new chaining bound for the Orlicz modulus of a stochastic process. 
The latter also applies to polyogonal partial sum processes of functional random variables and allows us to strengthen Donsker's invariance principle to all function spaces on the Besov-Orlicz scale up to the exact modulus with exponent 1/2. 
For the critical case, we establish the thresholded weak convergence of the Besov-Orlicz seminorm of the partial sum process. 
The analytical results find application in a nonparametric statistical testing problem, where Besov-Orlicz statistics are shown to detect a broader range of alternatives compared to Hölderian multiscale statistics. 

\end{abstract}

\section{Introduction}

For iid centered random variables $X_t$ with unit variance, we consider the interpolated, or polygonal, partial sum process 
\begin{align}
    S_n(u) = \frac{1}{\sqrt{n}} \sum_{t=1}^{\lfloor un\rfloor} X_t + \frac{un-\lfloor un\rfloor}{\sqrt{n}} X_{\lceil un\rceil}.\label{eqn:partialsum}
\end{align}
By Donsker's Theorem, this process converges to a standard Brownian motion $W$, weakly in the space $C[0,1]$ of continuous functions. 
This implies that $T(S_n)\overset{d}{\rightarrow} T(W)$ for any continuous functional $T$ on the space $C[0,1]$.
As both processes $S_n$ and $W$ admit more path regularity beyond mere continuity, the weak convergence can also be established in smaller function spaces, which allows for the asymptotic treatment of $T(S_n)$ for a broader class of functionals $T$.
The bottleneck to derive these invariance principles is typically a tightness conditions. 
Specifically, to show weak convergence $S_n\wconv W$ in a function space $\mathcal{F}$, one needs to show that for any $\epsilon>0$ there exists a compact set $K\subset \mathcal{F}$ such that $P(S_n\in K) \geq 1-\epsilon$. 
Compactness in function spaces is closely related to regularity, e.g.\ via the Arzela-Ascoli Theorem.
Thus, extensions of Donsker's invariance principle to smaller function spaces $\mathcal{F}$ demand sharp oscillations bounds for $W$ and $S_n$. 

What is the smallest function space where one can expect an invariance principle to hold?
A minimal requirement is that the Brownian paths lie in $\mathcal{F}$.
Lévy's result on the modulus of continuity of Brownian motion yields the limit 
\begin{align}
    \lim_{h\to 0}\sup_{u\in[0,1]} \frac{|W_{u+h}-W_u|}{\rho_{1/2,1/2}(h)}=\sqrt{2}, \qquad \text{for }\rho_{\mu,\nu}(h) = h^\mu |\log (e/h)|^\nu, \label{eqn:levy}
\end{align}
and hence the smallest Hölder-type space for the Brownian paths is $C^{\rho_{1/2,1/2}} = B^{\rho_{1/2,1/2}}_{\infty,\infty}$. 
On the other hand, the Brownian motion admits the sharper modulus of continuity $\rho(h)=\sqrt{h}$ if the regularity is measured not uniformly, but in a weaker norm. 
Indeed, \citet{ciesielski_modulus_1991} showed that the paths belong to the Besov space $B^{1/2}_{p,q}$ for any $p\geq 1$, and even to the Besov-Orlicz space $B^{1/2}_{\Psi_2,\infty}$; see Section \ref{sec:Besov} for a formal definition. Since  $B^{1/2}_{\Psi_2,\infty}\subsetneq C^{\rho_{1/2, 1/2}}$, this is a strict improvement on the earlier Hölder regularity, which also holds for Brownian motion in Banach spaces \cite{hytonen_besov_2008}.
However, unlike Lévy's result for the uniform modulus, the exact limiting behavior of the Orlicz modulus $\omega_{\Psi_2}(h,W)$ as $h\to 0$ has not been found so far.
Our first contribution is to show that 
\begin{align*}
    \lim_{h\to 0}\frac{\omega_{\Psi_2}(h,W)}{\sqrt{h}}\to \sqrt{\frac{8}{3}}, 
\end{align*}
and we also provide the corresponding rate of convergence. 
The limit is also derived for the vector-valued case.

These regularity results for the limiting Brownian motion provide the reference framework for the possible invariance principles.
For smoothness exponent $s<\frac{1}{2}$, the weak convergence of the partial sum process has been established in the Hölder spaces $C^s$ \citep{lamperti_convergence_1962,hamadouche_weak_1998,rackauskas_holder_2004,rackauskas_necessary_2004,rackauskas_convergence_2020}, in the Besov space $B_{p,q}^s$ for $p,q\in[1,\infty)$,\citep{morel_weak_2004}, and in $B_{p,\infty}^s$ \citep{giraudo_weak_2017}.
Note that $C^s\subsetneq B^s_{p,\infty}$, i.e.\ the switch to Besov spaces does not yield a sharper invariance principle.
Stronger smoothness moduli of the form $\rho_{1/2,\nu}(h) = \sqrt{h} |\log(e/h)|^\nu $ for $\nu>1/2$ have been investigated by \cite{rackauskas_necessary_2004}, establishing weak convergence of the partial sum process in $C^{\rho_{1/2,\nu}} \subsetneq C^s$.
Recently, it has been shown that Donsker's Theorem holds in the Hölder space $C^\rho = B^{\rho}_{\infty,\infty}$ for any modulus $\rho(h)\gg \rho_{1/2,1/2}(h)=\sqrt{h|\log (e/h)|}$ as $h\downarrow 0$ \citep{kohne_at_2025}.
This in particular includes cases such that $C^\rho \subsetneq C^{\rho_{1/2,\nu}}$ for all $\nu>1/2$.
In view of \eqref{eqn:levy}, the latter is essentially optimal.
To derive an invariance principle for even sharper moduli, we need to leave the Hölder scale of functions.
Currently, the only version of Donsker's Theorem for the Besov-Orlicz spaces is due to \citet{ait_ouahra_critere_2011, ait_ouahra_principe_2012} who treat the case, $B_{\Psi_2,\infty}^{\rho_{1/2,\nu}}$ for $\nu>1$.
Since $C^{\rho_{1/2, 1/2}} \hookrightarrow B^{\rho_{1/2,1}}_{\Psi_2,\infty} \hookrightarrow C^{\rho_{1/2,3/2}}$, see Proposition \ref{prop:embedding} below, the latter Besov invariance principle is in fact not as sharp as the Hölderian invariance principles of \cite{rackauskas_necessary_2004} and \cite{kohne_at_2025}, and hence sub-optimal. 
It should be noted that all mentioned results require suitable moment conditions on the random variables $X_t$ which we do not compare here.
In particular, we restrict our attention to the iid case and do not allow for nonstationarity or temporal dependence.

In this paper, we establish the weak convergence of $S_n$ in $B^{\rho}_{\Psi_2,\infty}$ for all $\rho(h)\gg \rho_{1/2,0}(h)=\sqrt{h}$. 
This is the sharpest invariance principle to date, and optimal on the Besov and Hölder scale of functions.
Moreover, the convergence is shown to hold for random variables in separable Banach spaces of type 2.
For the edge case $\rho(h) = \rho_{1/2,0}(h) = \sqrt{h}$, i.e.\ in the space $\B$, tightness can not be established, similar the critical Hölder space $C^{\rho_{1/2,1/2}}$. 
Instead, we use the recently developed concept of thresholded weak convergence \citep{kohne_at_2025} and show that, for some $\tau_0>0$,
\begin{align}
    (|S_n|_{\B} \vee \tau) \wconv (|W|_{\B}\vee \tau), \qquad \tau\geq \tau_0.\label{eqn:thresholded-Donsker}
\end{align}
The proof builds on a novel bound for the Orlicz modulus of continuity of stochastic processes with sub-Gaussian increments presented in Section \ref{sec:modulus}.
The latter is applicable to both $S_n$ and its limit $W$, and yields refined path properties of the Brownian motion (Section \ref{sec:Brownian}) as well as asymptotic equicontinuity of $S_n$, and thus weak convergence via Prokhorov's Theorem (Section \ref{sec:Donsker}).
Our results also apply to random variables $X_t$ taking values in a separable type 2 Banach space, and to vector-valued Brownian motion.

The new asymptotic theory for the partial sum process has implications in nonparametric statistics, where it allows to improve upon existing multiscale procedures for the signal discovery problem in certain regimes.
Based on the idealized white noise model, we demonstrate the advantage of using the Besov-Orlicz norm as test statistic instead of a Hölderian multiscale procedure (Section \ref{sec:signal-discovery}).
Subsequently, we employ the invariance principle to transfer this finding to the standard nonparametric regression model.
Simulations illustrate the gain in statistical power and highlight the impact of analytical results about the sample path regularity of Brownian motion on statistical methodology.

All technical proofs are gathered in the Appendix.

\paragraph{Notation} 
For a random variable $X$, we denote the sub-Gaussian Orlicz norm as $\|X\|_{\Psi_2} = \inf\{ K: \E \exp(X^2/K^2) \leq 2\}$, and $\Psi_2(x) = \exp(x^2)-1$.
For real numbers $a,b$, we denote $a\wedge b = \min(a,b)$ and $a\vee b = \max(a,b)$. For two sequences $a_n,b_n$, we denote $a_n\ll b_n$ to mean that $a_n/b_n\to 0$.

\section{Preliminaries on Besov-Orlicz spaces}\label{sec:Besov}

This section provides a brief background on Besov-Orlicz spaces.
Orlicz norms are common tools in probability to quantify the tails of a random variable, see \cite[Sec.~2.7.1]{Vershynin2003}. 
In functional analysis, Orlicz norms serve a similar purpose to generalize $L_p$ spaces.
Here, we deal with random functions and thus need both types of norms, which we distinguish notationally.
In the sequel, $(\mathcal{X},|\cdot|)$ denotes a normed vector space, which may be infinite dimensional.

\begin{definition}
    Let $\Psi:[0,\infty]\to[0,\infty]$ be a Young function, i.e.\ convex and increasing such that $\Psi(x)/x\to 0$ as $x\to 0$ and $\Psi(x)/x\to \infty$ as $x\to \infty$.\\
    For a random variable $X$ taking values in $\mathcal{X}$, define the Orlicz norm 
        \begin{align*}
            \|X\|_{\Psi} = \|X\|_{\Psi(dP)} = \inf\left\{K\;:\; \E \Psi(|X|/K) \leq 1 \right\}.
        \end{align*}
    For a function $f:[0,1]\to\mathcal{X}$, define the Orlicz norm
        \begin{align*}
            \|f\|_{\Psi} = \|f\|_{\Psi(dx)} = \inf\left\{K\;:\; \smallint \Psi(|f(x)|/K)\, dx \leq 1 \right\},
        \end{align*}
    and the corresponding modulus of continuity
    \begin{align*}
        \omega_\Psi(h,f) = \inf\left\{K\;:\; {\textstyle \int_0^{1-h}} \Psi\left(\tfrac{|f(x+h)-f(x)|}{K}\right)\, dx \leq 1 \right\}, \qquad h\in (0,1).
    \end{align*}
\end{definition}

As an extension of the classical Besov spaces, \cite{pick_several_1993} introduced the Besov-Orlicz spaces by measuring smoothness not in $L_p$, but in an arbitrary Orlicz norm. 
While \cite{pick_several_1993} only consider polynomial moduli $\rho(h)=h^s$, the extension to arbitrary smoothness is straightforward.

\begin{definition}\label{def:Besov-Orlicz}
    Let $\rho:[0,1]\to[0,\infty)$ be a modulus of continuity, i.e.\ increasing and continuous with $\rho(0)=0$,
    and let $\Psi:[0,\infty]\to[0,\infty]$ be a Young function.
    For any $f:[0,1]\to\mathcal{X}$ define the seminorm
    \begin{align}
        |f|_{B_{\Psi,\infty}^\rho} = \sup_{h\in(0,1)} \frac{\omega_{\Psi}(h,f)}{\rho(h)} \nonumber.
    \end{align}
    The space $B_{\Psi,\infty}^\rho$ consists of all functions $f$ such that $\|f\|_{B_{\Psi,\infty}^\rho} = \|f\|_{\Psi} + |f|_{B_{\Psi,\infty}^\rho}$ is finite, identifying functions which are identical almost everywhere.
    We use the shorthand $\B$ for $\Psi_2(x) = e^{x^2}-1$ and $\rho(h) = \sqrt{h}$, and $\Psi_p(x)=e^{x^p}-1$ for $p\geq 1$.
\end{definition}

Note that the Besov-Orlicz scale nests the usual Besov spaces via the polynomial Young function $\Psi(x)=x^p$.
Moreover, we can contrast this with the classical Hölder-type spaces $C^\rho$, which can be put on the Besov scale of function spaces as $C^{\rho}=B^\rho_{\infty,\infty}$ by replacing the Orlicz norm $\|\cdot\|_{\Psi(dx)}$ with the supremum norm $\|\cdot\|_{L_\infty(dx)}$, which may formally be interpreted as an Orlicz norm with Young function $\Psi(x) = \infty \cdot \mathds{1}(x>1)$.

\begin{definition}\label{def:Holder}
    Let $\rho:[0,1]\to[0,\infty)$ be a modulus of continuity, i.e.\ increasing and continuous with $\rho(0)=0$.
    The space $C^\rho$ consists of all functions $f:[0,1]\to\mathcal{X}$ such that the seminorm
    \begin{align}
        |f|_{C^{\rho}}=\sup_{x,y\in[0,1]} \frac{|f(x)-f(y)|}{\rho(|x-y|)}\nonumber
    \end{align}
    is finite. 
    We use the shorthand $C^{s}$ for $\rho(h) = h^s$.
\end{definition}

By the Arzela-Ascoli Theorem, the embedding $C^{\rho}\hookrightarrow C^{\rho'}$ is compact whenever the two moduli satisfy $\rho(h)\ll \rho'(h)$ as $h\to 0$.
For Besov-Orlicz spaces, the same property may be derived via a Wavelet basis representation due to \cite{ciesielski_quelques_1993}. 
We further extend this embedding to vector-valued Besov-Orlicz spaces.

\begin{proposition}[Besov-Orlicz embeddings]\label{prop:embedding}
    Let $\mathcal{X}$ be a separable Banach space, $\rho$ a modulus of continuity such that $ch\leq \rho(h) \leq C h^{r}$ for some $r>0$ and $0<c<C<\infty$, and $\rho^*(h)=\rho(h)\sqrt{|\log h|}$.
    Then there exists a continuous embedding
    \begin{align*}
        B^{\rho}_{\Psi_2,\infty} \hookrightarrow C^{\rho^*}([0,1]).
    \end{align*}
    Moreover, let $K\subset B^{\rho}_{\Psi_2,\infty}$ be bounded such that $K_t=\{f(t)\,:\, f\in K\}\subset \mathcal{X}$ is compact for every $t\in [0,1]$.
    Then $K$ is relatively compact in $B^{\rho'}_{\Psi_2,\infty}$ for every modulus $\rho'$ such that $\rho(h)\ll \rho'(h)$ as $h\downarrow 0$.
    In particular, 
    \begin{align*}
        B^{\rho}_{\Psi_2,\infty} \hookrightarrow B^{\rho'}_{\Psi_2,\infty} \qquad \text{compactly if }\mathcal{X}=\R.
    \end{align*}
\end{proposition}

Note that functions $f\in B^{\rho}_{\Psi_2,\infty}$ are only identified up to equivalence. 
Hence, the embedding $B^{\rho}_{\Psi_2,\infty}\hookrightarrow C^{\rho^*}$ consist of choosing one representative function in the equivalence class, which is continuous.

\section{Besov-Orlicz regularity of stochastic processes with sub-Gaussian increments}\label{sec:regularity}

\subsection{Modulus of continuity of stochastic processes with sub-Gaussian increments}\label{sec:modulus}

To establish tightness of the partial sum process $S_n$ via Prokhorov's Theorem and Proposition \ref{prop:embedding}, we need sharp control of the sub-Gaussian modulus of continuity of a stochastic process. 
Our main technical result is the following theorem.
We particularly want to highlight the construction of the chaining grid, see \eqref{eqn:chaining-raw-2} in the Appendix, which is crucial to achieve the sharp constant for the Gaussian case.
The grid construction is inspired by the proof of Lévy's limit for the ordinary modulus of Brownian motion as presented in \cite[10.3]{schilling_brownian_2021}. 

\begin{theorem}\label{thm:subgauss-int-concentrate}
    Let $(Z_u)_{u\in[0,1]}$ be a stochastic process taking values in a separable Banach space $(\mathcal{X}, |\cdot|)$ such that
    \begin{itemize}
        \item[(i)] $\|Z_{u}-Z_v\|_{\Psi_2} \leq \tau \sqrt{|u-v|}$ for all $u,v\in[0,1]$.
        \item[(ii)] For any $u_1<v_1$ and $u_2< v_2$ such that $v_1+h_0\leq u_2$, the increments $(Z_{v_1}-Z_{u_1})$ and $(Z_{v_2}-Z_{u_2})$ are independent. 
        \item[(iii)] $\omega_{\Psi_2}(h,Z)\leq \sqrt{h} \sqrt{\frac{h}{h_0}}R$ for all $h\leq h_0$, and a finite random variable $R$ with $P(R>r)\leq G(r) \to 0$ as $r\to \infty$.
    \end{itemize}
    Then for any $\epsilon>0$,
    \begin{align}
        P\left( \sup_{h\leq \delta} \frac{\omega_{\Psi_2}(h,Z)}{\sqrt{h}} > \tau(1+\epsilon) \right) \leq K(\tau,\epsilon, G, \delta) \overset{\delta\downarrow 0}{\longrightarrow} 0. \nonumber
    \end{align}
\end{theorem}

For the case that $Z_u$ is a Brownian motion, condition (ii) holds with $h_0=0$ and thus condition (iii) is vacuous. 
The latter condition is introduced to also cover the case of the partial sum process, where it is satisfied with $h_0=2/n$.
The difference is that for the Brownian case, we can use a chaining construction in time $u$ up to the finest resolutions due to self-similarity. 
For the partial sum process, this construction is only fruitful up to the discretization stepsize, below which the extra smoothness due to linear interpolation can be leveraged.

A consequence of Theorem \ref{thm:subgauss-int-concentrate} is that 
\begin{align}
    \limsup_{h\to 0} \frac{\omega_{\Psi_2}(h,Z)}{\sqrt{h}} \leq \tau, \label{eqn:omega-limit}
\end{align}
which is sharp for the case of vector-valued Brownian motion, see Theorem \ref{thm:modulus-Brownian} below.
Another question about the limit \eqref{eqn:omega-limit} is the corresponding rate of convergence. 
To this end, we need to slightly sharpen condition (i) as follows:
 \begin{itemize}
        \item[(i)'] $\E \exp\left(\lambda |Z_u-Z_v|\right) \leq  \kappa \exp\left( \frac{3}{16} \tau^2 \lambda^2 |u-v| \right) $ for all $u,v\in[0,1]$ and $\lambda\in\R$. 
\end{itemize}
\begin{remark}\label{rem:subgauss}
    For real-valued processes $Z_u$, it might appear more natural to consider the condition 
    \begin{align}
        \E \exp(\lambda(Z_u-Z_v)) \leq \exp\left( \frac{3}{16} \tau^2 \lambda^2 |u-v| \right), \label{eqn:variance-proxy}
    \end{align}
    which is a bound on the the sub-Gaussian variance-proxy of $Z_u-Z_v$, which is defined via the moment generating function as 
    \begin{align*}
        \sigma^*(X)=\inf\left\{ \sigma>0\,\big|\,\E \exp(\lambda X) \leq \exp\left(-\frac{\lambda^2}{2\sigma^2}\right)\;\text{for all }\lambda\right\}.
    \end{align*}
    This yields a norm which is equivalent to the sub-Gaussian norm $\|X\|_{\Psi_2}$, and \cite{leskela_sharp_2026} established the sharp inequality $\sqrt{3/8} \|X\|_{\Psi_2} \leq \sigma^*(X) \leq \sqrt{\log 2} \|X\|_{\Psi_2}$. 
    Hence, the factor in the right hand side of \eqref{eqn:variance-proxy} is chosen such that it implies condition (i) with the same $\tau$, and is indeed slightly stronger. 
    However, for the general vector-valued case, \eqref{eqn:variance-proxy} is not sensible and we instead consider the norm, which gives rise to the factor $\kappa$ in (i)'. 
    Due to this factor, (i)' no longer implies (i), but it nonetheless provides sufficiently strong tail bounds to improve upon Theorem \ref{thm:convergence}.   
    Note also that if $Z$ is a standard univariate Brownian motion, both conditions are satisfied with $\tau=\sqrt{8/3}$.
\end{remark}

Under the stronger assumption (i)', we can derive the following rate of convergence.

\begin{theorem}\label{thm:modulus-fine}
    Let $Z_u$ satisfy conditions (i)', (i), (ii), (iii).   
    Then for any $s<\frac{1}{10}$,
    \begin{align}
        \limsup_{h\to 0} h^{-s} \left[\frac{\omega_{\Psi_2}(h,Z)}{\sqrt{h}} - \tau \right] \;\leq \;0 \qquad \text{almost surely}.\nonumber 
    \end{align}
    This implies that
    \begin{align}
        \sup_{h\leq 1} h^{-s}\left[\frac{\omega_{\Psi_2}(h,Z)}{\sqrt{h}} - \tau \right] < \infty \qquad \text{almost surely}.\nonumber
    \end{align}
\end{theorem}

As part of the proof of Theorem \ref{thm:modulus-fine}, we derive and use the following tail bound for the modulus for a single $h$, which is of independent interest.

\begin{lemma}\label{lem:modulus-concentrate}
    Under conditions (i), (ii), (iii), there exists a $C>0$ such that for all $h\in (0,1)$, and all $\eta\in[0,2]$
    \begin{align*}
        P\left(\omega_{\Psi_2}(h,Z)\geq \tau \sqrt{h}(1+r) \right) \leq C(h\vee h_0)^{\eta r} \left( \frac{1}{r^2+(2-\eta)r} \right), \quad r\in[0,\tfrac{1}{2}].
    \end{align*}
    Under conditions (i)', (ii), and (iii), there exists for any $p\in(1,\frac{4}{3})$ a $C_p$ such that for all $h\in(0,1)$,
    \begin{align*}
        P\left(\omega_{\Psi_2}(h,Z)\geq \tau \sqrt{h}(1+r) \right) \leq C_p \left(\frac{ \kappa(1+r)}{r}\right)^{p} (h\vee h_0)^{p-1} ,\qquad r>0.
    \end{align*}
\end{lemma}

\begin{remark}
    Until the recent contribution of \citet{kempka_path_2024}, $\B$ was the smallest space to which the sample paths of the limiting Brownian motion $W$ are known to belong, and thus the smallest space in which we can expect some form of weak convergence of $S_n$. 
    Whether any distributional convergence also applies in the smaller space introduced by \citet{kempka_path_2024} remains open. 
\end{remark}

\subsection{Regularity of vector-valued Brownian motion}\label{sec:Brownian}

As a special case of the previous results, we consider $Z=W$ for a vector-valued Brownian motion $W$.
To be precise, we suppose that $W$ is a centered Brownian motion on some separable Banach space $(\mathcal{X},\|\cdot\|)$, i.e.\ $v\mapsto l^* W_v$ is a univariate Brownian motion for each continuous linear functional $l^*:\mathcal{X}\to \R$.
Extending the work of \cite{ciesielski_orlicz_1993} on univariate Brownian motion, it was shown by \citet{hytonen_besov_2008} that the vector-valued Brownian motion admits paths in $\B$.
However, the exact behavior of $\omega_{\Psi_2}(h,W)$ as $h\to 0$ has been studied neither for the vector-valued nor for the univariate case.
Via Theorem \ref{thm:subgauss-int-concentrate}, we can obtain explicit upper bounds on the latter modulus.
In Section \ref{sec:modulus}, lower bounds on the modulus could not be obtained because we only impose upper bounds on the stochastic process.
For the Brownian case, more structure is available, allowing us to derive matching lower bounds.
This yields the exact limit of the modulus of continuity, and also shows that Theorem \ref{thm:subgauss-int-concentrate} is in general sharp.

\begin{theorem}\label{thm:modulus-Brownian}
    Let $W$ be a centered Brownian motion on some separable Banach space $(\mathcal{X},|\cdot|)$, i.e.\ $v\mapsto l^*W(v)$ is a univariate Brownian motion for each continuous linear functional $l^*:\mathcal{X}\to \R$.
    Then,
    \begin{align*}
        \limsup_{h\to 0} \frac{\omega_{\Psi_2}(h,W)}{\sqrt{h}} \;=\; \|W(1)\|_{\Psi_2} \qquad \text{almost surely}. 
    \end{align*}
    For the univariate case $\mathcal{X}=\R$, we find 
    \begin{align*}
        \limsup_{h\to 0} \frac{\omega_{\Psi_2}(h,W)}{\sqrt{h}} \;=\; \sqrt{\tfrac{8}{3}} \qquad \text{almost surely}. 
    \end{align*}
\end{theorem}

For comparison, in the real-valued case, \cite{marcus_lp_2008} showed that the $L_p$ modulus converges as $\lim_{h\to 0} \omega_{L_p}(h, W)/\sqrt{h} = \E |W(1)|^p$ for any $p\geq 1$, and a corresponding central limit theorem corresponding is derived in \cite{marcus_clt_2008}.
The same limit for the $L_p$ modulus of Banach-space valued Brownian motion has been derived by \cite{hytonen_besov_2008}; see equation (4.1) therein. 
In contrast, the presented result on the sub-Gaussian modulus of continuity is novel.
We note that similar proof techniques could also yield the exact modulus in the exponential-type Orlicz norms $\Psi_p(x)=\exp(x^p)-1$, with $p\in[1,2)$.

For the univariate case, we also obtain rates of convergence for the modulus of continuity as a consequence of Theorem \ref{thm:modulus-fine} and Lemma \ref{lem:modulus-concentrate}.

\begin{theorem}\label{thm:Brownian-concentration}
    Let $W$ be a standard Brownian motion taking values in $\R$.
    Then, for any $s<\frac{1}{4}$, and any $s^*<\frac{1}{10}$,
    \begin{align}
        \max\left[0,\;\frac{\omega_{\Psi_2}(h,W)}{\sqrt{h}} - \sqrt{\frac{8}{3}}\right] \;=\; \begin{cases}
        \mathcal{O}_{P}\left( h^{s} \right), \\
        \mathcal{O}_{\mathrm{a.s.}}\left( h^{s^*} \right)\quad \text{almost surely}. 
\end{cases}\label{eqn:rate-1}
    \end{align}
\end{theorem}

We conjecture that the faster rate in probability could be sharpened to the boundary case $s=\frac{1}{4}$, but whether the second convergence rate can be improved to match the faster rate is unclear.

To derive the rate of convergence for the vector-valued case, we would require a quantitative tail bound as in (i)' for the Brownian motion in Banach spaces.
Fernique's Theorem states that $\E\exp(\eta \|W(1)\|^2)<\infty$ for some $\eta>0$, see e.g.\ \cite[Thm.~2.7]{da_prato_stochastic_2014}, and thus (i)' does hold for some $\tau$.
However, this $\tau$ is in general not explicit and it is unknown (to us) whether $\tau=\|W(1)\|_{\Psi_2}$.
Thus, we leave the extension of Theorem \ref{thm:Brownian-concentration} to the vector-valued case open.

\section{Functional central limit theorems}\label{sec:Donsker}

We now return to the asymptotics of the polygonal partial sum process $S_n$.
For iid real-valued random variables $Y_t$, Donsker's Theorem yields weak convergence in $C[0,1]$.
The standard approach to prove this is to first establish the finite dimensional weak convergence of $(S_n(u_1),\ldots, S_n(u_m))$ for any $u_1,\ldots, u_m\in[0,1]$, and then extend this to functional weak convergence via tightness of the process $u\mapsto S_n(u)$. 
If the $Y_t$ are sub-Gaussian, Theorem \ref{thm:subgauss-int-concentrate} shows that $S_n\in \B$ almost surely, and the compact embedding $\B\hookrightarrow C[0,1]$ yields tightness.
In view of the compact Besov-Orlicz embedding $\B\hookrightarrow B^\rho_{\Psi_2,\infty}$ for any $\rho(h)\gg \sqrt{h}$, see Proposition \ref{prop:embedding}, this also yields tightness in even smaller spaces, and thus gives rise to sharper functional central limit theorems. 

As the regularity result is established for normed spaces, the corresponding invariance principle also holds for functional data.
Here, we specifically assume $(\mathcal{X},\|\cdot\|)$ to be a separable Banach space of type 2. 
That is, there exists a factor $K=K(\mathcal{X})$ such that for any $x_1,\ldots, x_n\in\mathcal{X}$, and iid Rademacher random variables $\epsilon_i$, we have $\E\|\sum_{i=1}^n \epsilon_i x_i\|^2 \leq K \sum_{i=1}^n\|x_i\|^2$.
This includes all Hilbert spaces, in particular $\R^d$, as well as any $L_p$ function space with $p\geq 2$, see \cite[Sec.~9.2]{ledoux_probability_2011}.
This restriction ensures that a sub-Gaussian concentration inequality for sums of iid random variables holds, see Lemma \ref{lem:subgauss-Banach}, and that the central limit theorem holds for any iid sequence of square-integrable random variables \cite{hoffmann-jorgensen_law_1976}, see also \cite[Thm.~10.5]{ledoux_probability_2011}.
Beyond the class of type 2 Banach spaces, more structural assumptions on the random variables $X_i$ are required to establish a central limit theorem. 
For example, the Banach space $C[0,1]$ is not of type 2, and \cite{kutta_multiscale_2025} require additional Hölder smoothness assumptions of the random elements $X_t\in C[0,1]$ to establish weak convergence.

\begin{theorem}\label{thm:FCLT}
    Let $X_t$ be an iid sequence of centered random variables with values in a separable Banach space $(\mathcal{X},\|\cdot\|)$ of type 2, such that $\|X_t\|_{\Psi_2}<\infty$.
    Then $S_n\wconv W$ in $B^{\rho}_{\Psi_2,\infty}$ for any modulus $\rho(h)\gg \sqrt{h}$ as $h\to \infty$, where $W(u)_{u\in[0,1]}$ is a $\mathcal{X}$-valued Brownian motion with $\Cov(W(1))=\Cov(X_1)$.
\end{theorem}

Because of the continuous embedding $B^{\rho}_{\Psi_2,\infty} \hookrightarrow C^{\rho^*}$ with $\rho^*(h) = \rho(h) \sqrt{|\log h|}$, see Proposition~\ref{prop:embedding}, we also obtain a corresponding Hölderian invariance principle.

\begin{cor}\label{cor:FCLT}
    Under the conditions of Theorem \ref{thm:FCLT}, the weak convergence $S_n\wconv W$ holds in $C^{\rho^*}$ for any modulus $\rho^*(h) \gg \sqrt{h |\log h|}$.
\end{cor}

This invariance principle in Besov spaces is the first result for moduli of continuity in the regime $\sqrt{h}\ll \rho(h)\ll \sqrt{h |\log h|}$.
In comparison to the existing functional limit theorems in the literature, which require moduli $\rho(h) \gg \sqrt{h|\log h|}$, our stronger convergence result imposes slightly stricter assumptions. 
Specifically, we consider partial sums of iid sub-Gaussian random variables, whereas \cite{morel_weak_2004} and \cite{giraudo_weak_2017} have weaker moment assumptions, and \citet{ait_ouahra_critere_2011, ait_ouahra_principe_2012} show tightness of any stochastic process with sub-Gaussian increments, of which the partial sum process is an example.
While these results only cover real-valued random variables $X_t$, some authors consider Hölderian invariance principles for vector-valued (functional) random variables $X_t$. 
In particular, \cite{rackauskas_holderian_2009} study linear time series in Hilbert spaces and subsequently extend this to Banach spaces \citep{rackauskas_limit_2010}, though only for moduli $\rho(h)\gg \rho_{1/2,\nu}$ for some $\nu>1/2$.
We also refer to \cite{kutta_monitoring_2025} for $m$-approximable time series with values in Banach spaces, and to \cite{kutta_multiscale_2025} for mixing time series taking values in the space $C[0,1]$, both for moduli $\rho(h)=h^\mu$ with $\mu<1/2$.

The modulus $\rho(h)=\rho_{1/2,0}(h)=\sqrt{h}$ is critical as this is exactly the Besov-Orlicz regularity of the limiting Brownian motion. 
In this edge case, we can not establish weak convergence of the partial sum process, but it is nonetheless possible to obtain the distributional limit of the Besov-Orlicz seminorm.

\begin{theorem}\label{thm:convergence}
    Under the conditions of Theorem \ref{thm:FCLT}, there exists a $K=K(\mathcal{X})>0$ such that
    \begin{align}
        \left(|S_n|_{\B} \vee \tau\right) \;\wconv\; \left(|W|_{\B} \vee \tau\right), \quad \text{for all } \; \tau>K\|X\|_{\Psi_2}.\label{eqn:thresholded}
    \end{align}
    If $\mathcal{X}=\R$ and $\E \exp(\lambda X_i)\leq \exp(\sigma^2\lambda^2/2)$ for all $\lambda$, then \eqref{eqn:thresholded} holds for $\tau>\sigma\sqrt{8/3}$.
\end{theorem}

Theorem \ref{thm:convergence} shows that the Besov-Orlicz seminorm of $S_n$ converges in distribution, though not on its full support but only beyond a certain value $\tau$.
Note that the $\tau$ does not depend on $n$. 
This type of thresholded weak convergence has been discovered recently \citep{kohne_at_2025} when studying the Hölderian seminorm $|S_n|_{C^{\rho_{1/2,1/2}}}$ for $\rho_{1/2,1/2}(h)=\sqrt{h(\log e/h)}$, which is the regularity of the limiting Brownian motion in the Hölderian framework.
The present result allows for a sharper modulus of continuity as smoothness is measured in the functional sub-Gaussian norm, rather than uniformly. 
Moreover, the prior work \cite{kohne_at_2025} did not study infinite-dimensional random variables $X_i$.

\begin{example}\label{ex:gauss}
    Standard Gaussian random variables $X_t\sim\mathcal{N}(0,1)$ satisfy the condition of Theorem \ref{thm:convergence} with $\sigma=1$, and the threshold $\tau=\sqrt{8/3}$ matches exactly the limit in Theorem \ref{thm:modulus-Brownian}.
    In other words, $\|W\|_{\B}\geq \tau$ almost surely, and hence the convergence holds indeed on the whole support.
    In this sense, Theorem \ref{thm:convergence} is sharp.
    This also justifies to approximate $\|W\|_{\B}$ by $\|S_n\|_{\B}$ with iid standard Gaussian $X_i$.
\end{example}

\begin{example}\label{ex:cdf}
    The asymptotic results for the partial sum process can be applied to various functionals of interest in mathematical statistic. 
    We consider a recent example in the context of nonparametric changepoint detection.
    Let $Y_i$ be iid real-valued random vectors, and let $X_i = \mathds{1}(\cdot \leq Y_i)$.
    Then $\frac{1}{n}\sum_{i=1}^n X_i$ is the empirical distribution function, and $S_n$ the corresponding partial sum process. 
    If we consider $X_i$ as elements in the weighted $L_2$ space $L_2(\mu)$ for a finite measure $\mu$ on $\R^d$, then $X_i$ are bounded by one. 
    Since $L_2(\mu)$ is a Hilbert space, Theorem \ref{thm:FCLT} and \ref{thm:convergence} apply.
    This partial sum process has been employed by \cite{kutta_multiscale_2025} to design the multiscale changepoint test statistic
    \begin{align*}
        T_\rho(S_n)=\sup_{h\in[0,\frac{1}{2}]} \sup_{u\in [h,1-h]} \frac{\|[S_n(u)-S_n(u-h)] - [S_n(u+h)-S_n(u)]\|}{\rho(h)}.
    \end{align*}
    For any $\rho\gg \rho_{1/2,1/2}$, Theorem \ref{thm:FCLT} yields the weak convergence $S_n\wconv W$ in $C^{\rho}$ and thus $T_\rho(S_n)\wconv T_\rho(W)$. 
    This also covers moduli in the regime $\rho_{1/2,1/2}(h) \ll \rho(h)\ll \rho_{1/2,\nu}(h)$, which was previously excluded in the literature. 
    Inspired by the new Besov-Orlicz limit theory, we may also use the statistic
    \begin{align*}
        T_{\rho, \Psi_2}(S_n) = \inf\left\{ K\,:\, \sup_{0<h<\frac12} \int_{h}^{1-h} \Psi_2\left(\tfrac{\|[S_n(u)-S_n(u-h)] - [S_n(u+h)-S_n(u)]\|}{\rho(h)}\right)\, du \leq 1\right\}.
    \end{align*}
    As $T_{\rho,\Psi_2}$ is a continuous functional on $B^{\rho}_{\Psi_2,\infty}$, Theorem \ref{thm:FCLT} yields $T_{\rho,\Psi_2}(S_n)\wconv T_{\rho,\Psi_2}(W)$ for any $\rho(h)\gg \sqrt{h}=\rho_{1/2,0}$. 
    Moreover, arguments similar to Theorem \ref{thm:convergence} show that there exists some $\tau$ such that
    \begin{align*}
        \left(T_{\rho_{1/2,0}, \Psi_2} (S_n) \vee \tau \right) \quad\wconv\quad \left(T_{\rho_{1/2,0}, \Psi_2} (W) \vee \tau \right).
    \end{align*}
    Unlike the methodology presented by \cite{kutta_multiscale_2025}, our asymptotic theory currently does not allow for temporally dependent noise. 
\end{example}

\begin{example}
    Exact evaluation of the Besov-Orlicz seminorm of $S_n$ is not possible in practice and one needs to resort to discretization of the involved integrals.
    Instead, it is practically more feasible to consider the non-interpolated partial sum process $\widetilde{S}_n(u)=\frac{1}{\sqrt{n}} \sum_{t=1}^{\lfloor un\rfloor} X_t$, and the statistic 
    \begin{align*}
        D_{\rho,n}(\widetilde{S}_n) 
        &= \max_{m=1,\ldots, n} \omega_{\Psi_2}(\tfrac{m}{n}, \widetilde{S}_n)/\rho(m/n) \\
        &= \max_{m=1,\ldots, n} \inf\left\{K_m>0\,:\, \frac{1}{n} \sum_{a=0}^{n-m} \Psi_2\left( \frac{ \|\sum_{t=a+1}^{a+m} X_t\|  }{K_m \cdot \sqrt{n}}\right) \leq 1 \right\} \Big/ \rho(m/n),
    \end{align*}
    or the dyadic version
    \begin{align*}
        D_{\rho,n, \textsc{dyadic}}(\widetilde{S}_n) = \max_{m=\lceil 2^{-j} n\rceil, j\in \N_0} \omega_{\Psi_2}(\tfrac{m}{n}, \widetilde{S}_n)/\rho(m/n).
    \end{align*}   
    The discretized but interpolated versions $D_{\rho,n}(S_n)$ and $D_{\rho,n,\textsc{dyadic}}(S_n)$ can be defined analogously.
    Despite the discretization, these statistics admit the same limit distribution as the continuous variants
    \begin{align*}
        D_\rho(S_n)=|S_n|_{B^{\rho}_{\Psi_2,\infty}} 
            \quad \text{and}\quad
        D_{\rho,\textsc{dyadic}}(S_n) = \sup_{j\in \N_0} \omega_{\Psi_2}(2^{-j},S_n) / \rho(2^{-j}).
    \end{align*}
\end{example}

\begin{theorem}\label{eqn:CLT-discretized}
    Under the conditions of Theorem \ref{thm:FCLT}, 
    \begin{align*}
        D_{\rho,n}(S_n),\; D_{\rho,n}(\widetilde{S}_n),\; D_{\rho}(S_n)&\quad\wconv\quad D_\rho(W) ,
    \end{align*}
    i.e.\ all three statistics converge to the same limit distribution. 
    Similarly
    \begin{align*}
        D_{\rho,n,\textsc{dyadic}}(S_n),\;D_{\rho,n,\textsc{dyadic}}(S_n),\; D_{\rho,\textsc{dyadic}}(S_n)&\quad\wconv\quad D_{\rho,\textsc{dyadic}}(W).
    \end{align*}
    Moreover, for the same $K=K(\mathcal{X})$ as in Theorem \ref{thm:convergence}, and for the critical modulus $\rho(h)=\sqrt{h}$, these statistics converge beyond any threshold $\tau>(K+2)\|X\|_{\Psi_2}$. 
\end{theorem}

\section{Relation to the signal discovery problem}\label{sec:signal-discovery}

As a particular application of our analytical results in mathematical statistic, we consider the nonparametric regression model 
\begin{align}
    Y_t = f(\tfrac{t}{n}) + \eta_t, \qquad t=1,\ldots, n, \label{eqn:model-regression}
\end{align}
for centered, sub-Gaussian random variables $\eta_t$ with unit variance.
When performing inference on the unknown function $f$, a  fundamental problem to test the simple null hypothesis $\Hyp_0:f=f_0$ for a given $f_0$.
We focus without loss of generality on the canonical case $f_0\equiv 0$, which is also also known as the signal discovery problem, see \cite{spokoiny_adaptive_1998, walther_calibrating_2022} and \cite[Sec.~6.2.3]{Gine2016}.
Via test inversion for all candidate signals $f$ in nonparametric class of functions $\F$, statistical methods for this fundamental signal discovery problem readily enable goodness-of-fit testing, nonparametric confidence intervals, and localization of structural breaks, as outlined elsewhere \citep{kohne_at_2025}. 
Hence, the goal is to develop a test which is consistent against a rich nonparametric class of alternative signals $f$.

As a motivating example, consider the sequence of local alternatives $f_n(u) = \delta_n \mathds{1}(u\in [a_n, b_n]) $ for $0\leq a_n < b_n\leq 1$ and $\delta_n>0$, $|b_n-a_n|\geq \frac{1}{n}$.
This alternative can be detected with the global scan statistic $\max_I T_n(I)$ where $T_n(I) = |\sum_{t\in I} Y_t|/\sqrt{I}$ for any integer interval $I\subset [1,n]$. 
Assuming standard Gaussian errors, and upon choosing a suitable critical value of order $\sqrt{\log n}$ based on an extreme value limit distribution, the latter test is consistent in the regime $n \delta_n^2 |b_n-a_n| > 2\log(n)$ \citep{sharpnack_exact_2016}. 
A better detection threshold is achieved by the multiscale scan statistic of \citet{dumbgen_multiscale_2001} of the form $T_n^{\textsc{ds}} = \max_I T_n(I)-\sqrt{2\log (en/|I|)}$, which yields a consistent test if $n \delta_n^2 |b_n-a_n| > 2\log(e/|b_n-a_n|)$.
Since its inception, the concept of multiscale scan statistics has been exploited and developed by various researchers, e.g.\ for changepoint detection \citep{frick2014multiscale} and density estimation \citep{dumbgen_multiscale_2008}.
For non-Gaussian errors, the distribution of $T_n^{\textsc{ds}}$ can be approximated by its Gaussian analogue, but only if a lower bound on the interval length $I$ is imposed \citep{frick2014multiscale, dette_multiscale_2020}.
Thus, as discussed in \cite{kohne_at_2025}, the method does not achieve the minimax rate of detection against very short signals. 
Instead, it is argued therein that the multiplicatively weighted multiscale statistic
\begin{align*}
    T_n^{\textsc{km}} = \max_{I \subset [1,n]} \frac{T_n(I)}{\sqrt{\log (en/|I|)}},
\end{align*}
is more robust, as the central limit theorem applies to $T_n^{\textsc{km}}$ without a lower bound on the interval length. 
Indeed, the statistic $T_n^{\textsc{km}}$ yields a consistent test if $n \delta_n^2 |b_n-a_n| \gg \log(e/|b_n-a_n|)$, where the slightly stronger constraint is the price of not assuming Gaussianity.

A fundamental limitation of these supremum-type statistics is that they only use local information.
If there are many disjoint intervals $I$ for which $T_n(I)$ almost exceeds the critical value, the multiscale test will still not reject the null. 

\begin{proposition}\label{prop:sin}
    Let the errors $\eta_t$ be standard Gaussian, and consider the sequence of alternatives $f_n = \delta_n \sin(u/l_n)$ with $\delta_n, l_n\to 0$ such that $1\ll n \delta_n^2 l_n \ll \log(1/l_n)$.
    Then $T_n^{\textsc{ds}}$ and $T_n^{\textsc{km}}$ are stochastically bounded, i.e.\ the test fails to be consistent.
\end{proposition}

To connect this statistical problem to the regularity of Brownian motion, we note that model \eqref{eqn:model-regression} is equivalent to observing $\frac{1}{\sqrt{n}}\sum_{t=1}^k Y_t = \frac{1}{\sqrt{n}}\sum_{t=1}^k f(t) + \frac{1}{\sqrt{n}} \sum_{t=1}^k \eta_t$ for $k=1,\ldots, n$, which closely resembles the idealized Gaussian white noise model 
\begin{align}
    Y_n(u) = \sqrt{n}\int_0^u f(v)\ dv + W_v, \quad u\in[0,1], \label{eqn:GWN}
\end{align}
where $W_v$ is a Brownian motion; see \citet[Sec.~1.10]{Tsybakov2008}. 
In this setting, the statistic $T_n^{\textsc{km}}$ corresponds to the Hölder seminorm $|Y|_{C^{\rho_{1/2}}} = \sup_{u,v\in[0,1]} |Y(u)-Y(v)|/\rho_{1/2}(|u-v|) $ with $\rho_{1/2}(h)=\sqrt{h\log(e/h)}$, and the respective hypothesis test is consistent against any sequence of alternatives $f_n$ such that $|\sqrt{n}F_n|_{C^{\rho_{1/2}}}\to\infty$, where $F_n(u)=\int_0^u f_n(v)\, dv$.
For $f_n=\delta_n \sin(u/l_n)$ as in Proposition \ref{prop:sin}, we have $|F_n|_{C^{\rho_{1/2}}} \propto \sqrt{\delta_n^2 l_n / \log(1/l_n)}$, which reveals the regime where the test is inconsistent.

By moving from supremum-type statistics to functionals based on Orlicz norms, it is possible to make better use of the global structure of the signal.
Thus, we suggest to test for $f\equiv 0$ via the test statistic $|Y|_{\B}$, and use the quantiles of the random variable $|W|_{\B}$ as critical values. 
This yields a consistent test if $|\sqrt{n}F_n|_{\B}\to \infty$. 
Because the embedding $C^{\rho_{1/2}}\hookrightarrow \B$ is strict, this statistic is powerful against a broader range of alternatives $f_n$ where $|\sqrt{n}F_n|_{\B}\to \infty$ while $|\sqrt{n}F_n|_{C^{\rho_{1/2}}}\to 0$.
This specifically helps for the signals introduced above, without making explicit use of the periodicity.
It turns out that the nonparametric Besov-Orlicz test detects the sinusoidal alternative if $n \delta_n^2l_n\gg 1$, where classical multiscale tests fail.
This is a consequence of the following more general result about lower bounds on Besov-Orlicz norms.

\begin{proposition}\label{prop:lb-signal-norm-2}
    Let $f:[0,1]\to \R$ and $I_1,\ldots, I_m \subset [0,1]$ be disjoint intervals such that $|I_k|\geq 2 h$ for each $k$, and either $f(u)\geq \delta$ for all $u\in I_k$, or $f(u)\leq -\delta$ for all $u\in I_k$.
    Then
    \begin{align*}
        |F|_{B^\rho_{\Psi,\infty}} \geq \frac{\delta h}{\rho(h) \Psi^{-1}(\tfrac{1}{mh})}.
    \end{align*}
    In particular, for $\rho_\nu(h) = \sqrt{h} |\log h|^{1/\nu}$,
    \begin{align*}
    	|F|_{\B} &\geq  \delta \sqrt{h} |\log \tfrac{e}{mh}|^{-\frac{1}{2}},
    	\qquad
    	&|F|_{B^{\rho_\nu}_{\infty,\infty}} &\geq  \delta \sqrt{h} |\log \tfrac{e}{h}|^{-\frac{1}{\nu}},\\
    	|F|_{B^{1/2}_{p,\infty}} &\geq \delta \sqrt{h} (mh)^{\frac{1}{p}},
    	&|F|_{B^{\rho_\nu}_{\Psi_p,\infty}} 
    	&\geq 	
    	\delta \sqrt{h} |\log\tfrac{e}{h}|^{-\frac{1}{\nu}} |\log \tfrac{e}{mh}|^{-\frac{1}{p}}.
    \end{align*}
\end{proposition}

To compare the Besov-Orlicz statistic with the usual multiscale statistic, we contrast the lower bounds for the seminorms of $F$ in $\B$ and $B_{\infty,\infty}^{\rho_2}$ and find that the former is bigger if $m$ is large. 
For $m\propto 1/h$, the difference amounts to a factor $\sqrt{|\log h|}$ which enables the detectability of the sine signal in the regime $1\ll n \delta_n^2 l_n \ll \log(1/l_n)$.
We highlight that the sine signal only serves as illustration, while Proposition \ref{prop:lb-signal-norm-2} illustrates the consistency of the procedure against a broader range of alternatives. 
In particular, our test is nonparametric and does not use the specific sinusoidal shape.

To apply these ideas beyond the idealized white noise model, and rather to the original regression model \eqref{eqn:model-regression}, we consider the interpolated partial sum process $S_n(u)$ and use the Besov-Orlicz test statistic $T_n^{\textsc{bo}} = |\sqrt{n}S_n|_{\B}$. 
By virtue of Theorem \ref{thm:convergence}, $T_n^{\textsc{bo}}$ converges in distribution to the Besov-Orlicz seminorm of $W$. 
That is, our new functional central limit theorem makes the connection between the regression model \eqref{eqn:model-regression} and the idealized white noise model \eqref{eqn:GWN} precise.

We illustrate the difference of using the Besov-Orlicz versus the Hölder as statistic via a simulation.
First, we consider the power of the two test statistics in the white-noise model.
For the simulation, we choose a discretization grid of mesh size $10^{-5}$ for the discretization of the Brownian motion, the integrals for the Besov-Orlicz norm, and as evaluation points of the Hölder norm.
Moreover, to save computational time, we evaluate the supremum $h\in(0,1)$ only on a dyadic grid.
With this setup, each test statistic is simulated multiple times under the null, i.e.\ without a signal, and under the alternative, i.e.\ with a signal. 
The p-values under the alternative are computed based on the simulation of the null case, and thus corrected for discretization effects.

To show that the benefits of the Besov-Orlicz statistic are not restricted to the sine function, we consider the following signal, inspired by Proposition \ref{prop:lb-signal-norm-2}:
For a length scale $l$ and a magnitude $\delta$, we set 
\begin{align*}
    f_{\textsc{flip}}(x) = \sum_{k=0}^{\lceil 1/(2l)\rceil} \eta_k \left[\mathds{1}(x\in[2kl, (2k+1)l) -  \mathds{1}(x\in[(2k+1)l, 2(k+1)l)\right],
\end{align*}
where $\eta_k$ are iid Rademacher random variables, see
Figure \ref{fig:flipping_power}(a).
In every simulation round, the signs $\eta_k$ are sampled independently. 
For each lengthscale, the magnitude is chosen such that $n\delta^2 l = |\log_2 l|^{1/4}$.
Thus, as $l\to 0$, Proposition \ref{prop:sin} and \ref{prop:lb-signal-norm-2} suggests that the Besov-Orlicz statistic should be consistent, while the Hölder statistic is inconsistent.
This is in line with the empirical performance as shown in Figure \ref{fig:flipping_power}, which demonstrates the power of the test as a function of the significance level, i.e.\ the ROC curve.

\begin{figure}
    \centering
    \begin{subfigure}[T]{0.45\textwidth}
         \centering
         \includegraphics[width=\textwidth]{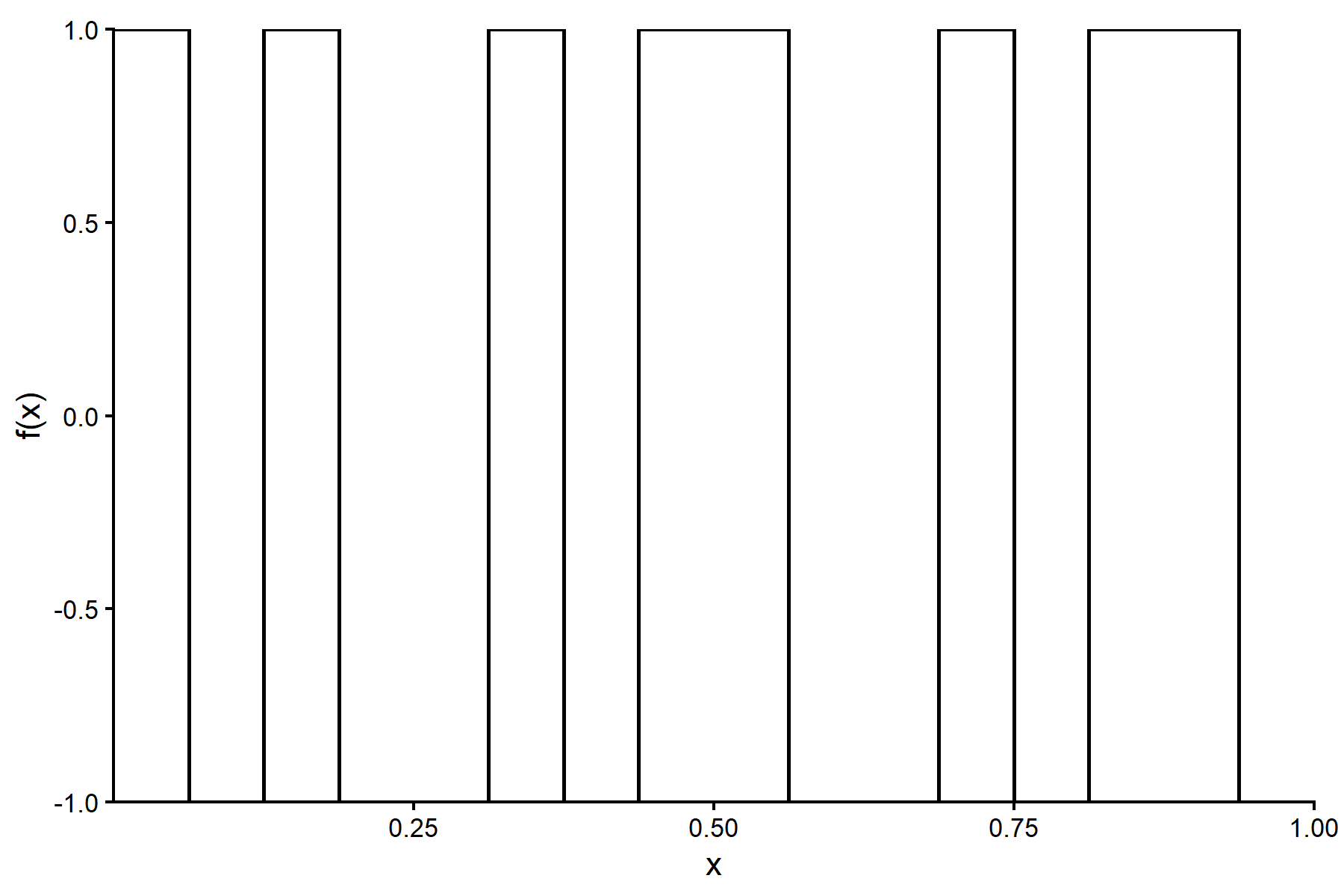}
         \caption{Signal $f_{\textsc{flip}}$ with length scale $l=2^{-4}$.}
         \label{fig:y equals x}
     \end{subfigure}
     \begin{subfigure}[T]{0.45\textwidth}
         \centering
         \includegraphics[width=\textwidth]{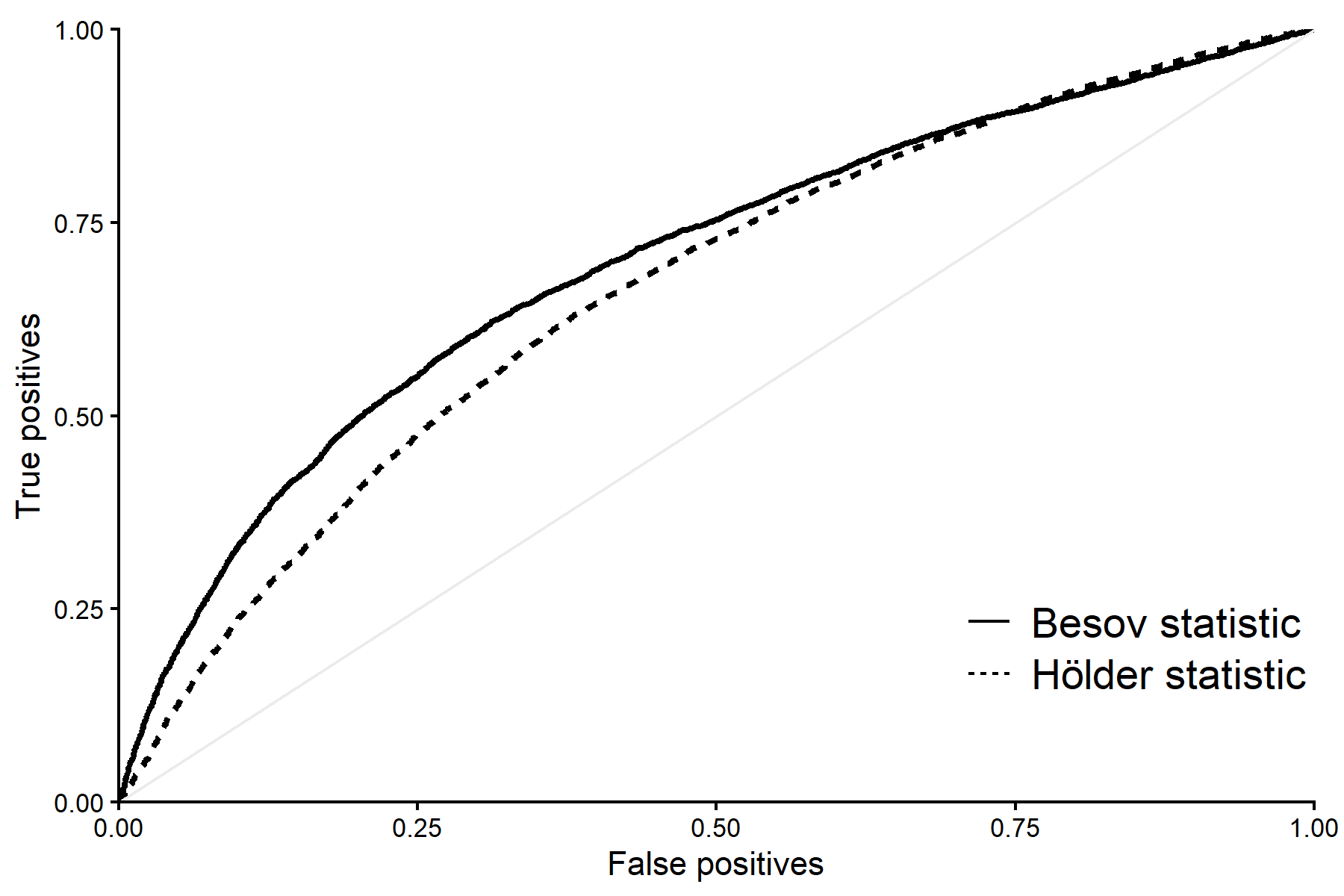}
         \caption{Power against $f_{\textsc{flip}}$ with $l=2^{-2}$.}
     \end{subfigure}
     \begin{subfigure}[T]{0.45\textwidth}
         \centering
         \includegraphics[width=\textwidth]{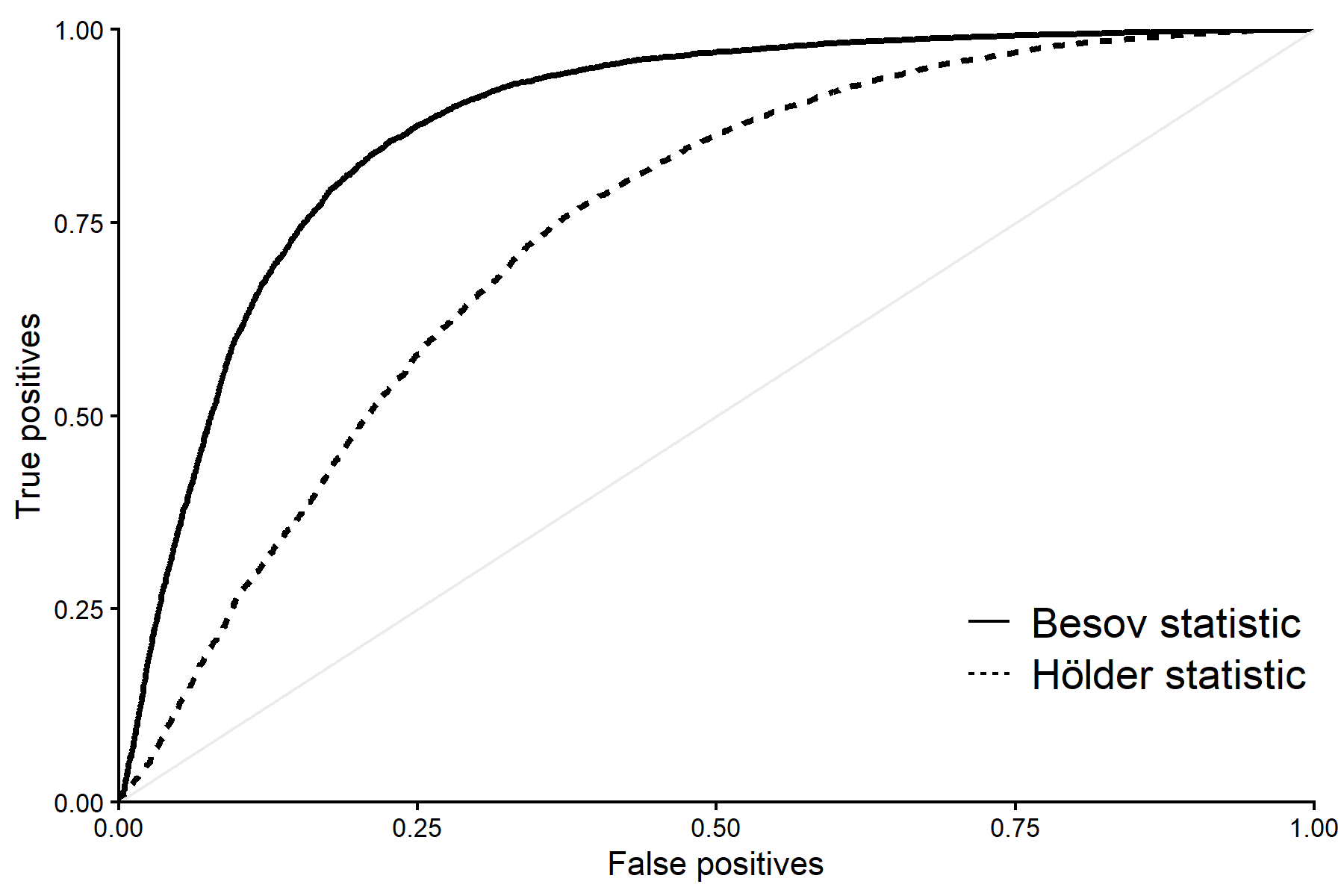}
         \caption{Power against $f_{\textsc{flip}}$ with $l=2^{-4}$.}
     \end{subfigure}
     \begin{subfigure}[T]{0.45\textwidth}
         \centering
         \includegraphics[width=\textwidth]{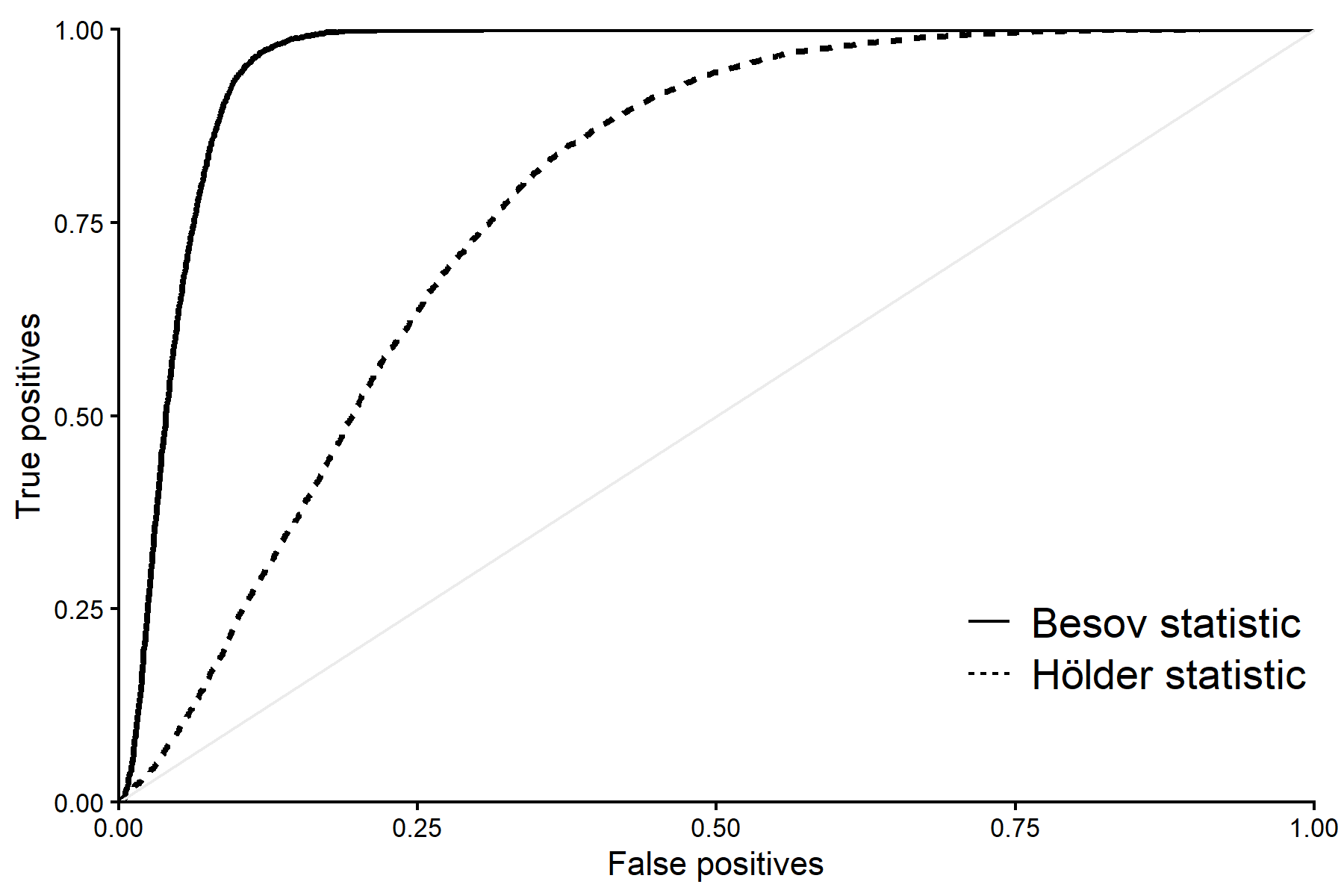}
         \caption{Power against $f_{\textsc{flip}}$ with $l=2^{-6}$.}
     \end{subfigure}
    \caption{Example test signal (a) and ROC curves (b-d) for the Besov-Orlicz statistic $\|Y_n\|_{\B}$ and the Hölder statistic $\|Y_n\|_{C^{\rho_{1/2}}}$ in the white noise model, with alternatives $f_{\textsc{flip}}$. Reported quantities are based on $10^4$ Monte Carlo simulations.}
    \label{fig:flipping_power}
\end{figure}

As a second example, we consider the Doppler function $f_{\textsc{doppler}}(x) = \delta  \sin(4/x)$ for different $\sqrt{n}\delta\in\{2, 4, 8\}$, see Figure \ref{fig:doppler_power}(a).
In this case, both statistics perform comparably.
This shows that the Besov-Orlicz statistic yields stronger tests is certain regimes, and generally comparable tests, which reflects the fact that $\B\hookrightarrow C^{\rho_{1/2}}$ is a strict, continuous embedding.

\begin{figure}
    \centering
    \begin{subfigure}[T]{0.45\textwidth}
         \centering
         \includegraphics[width=\textwidth]{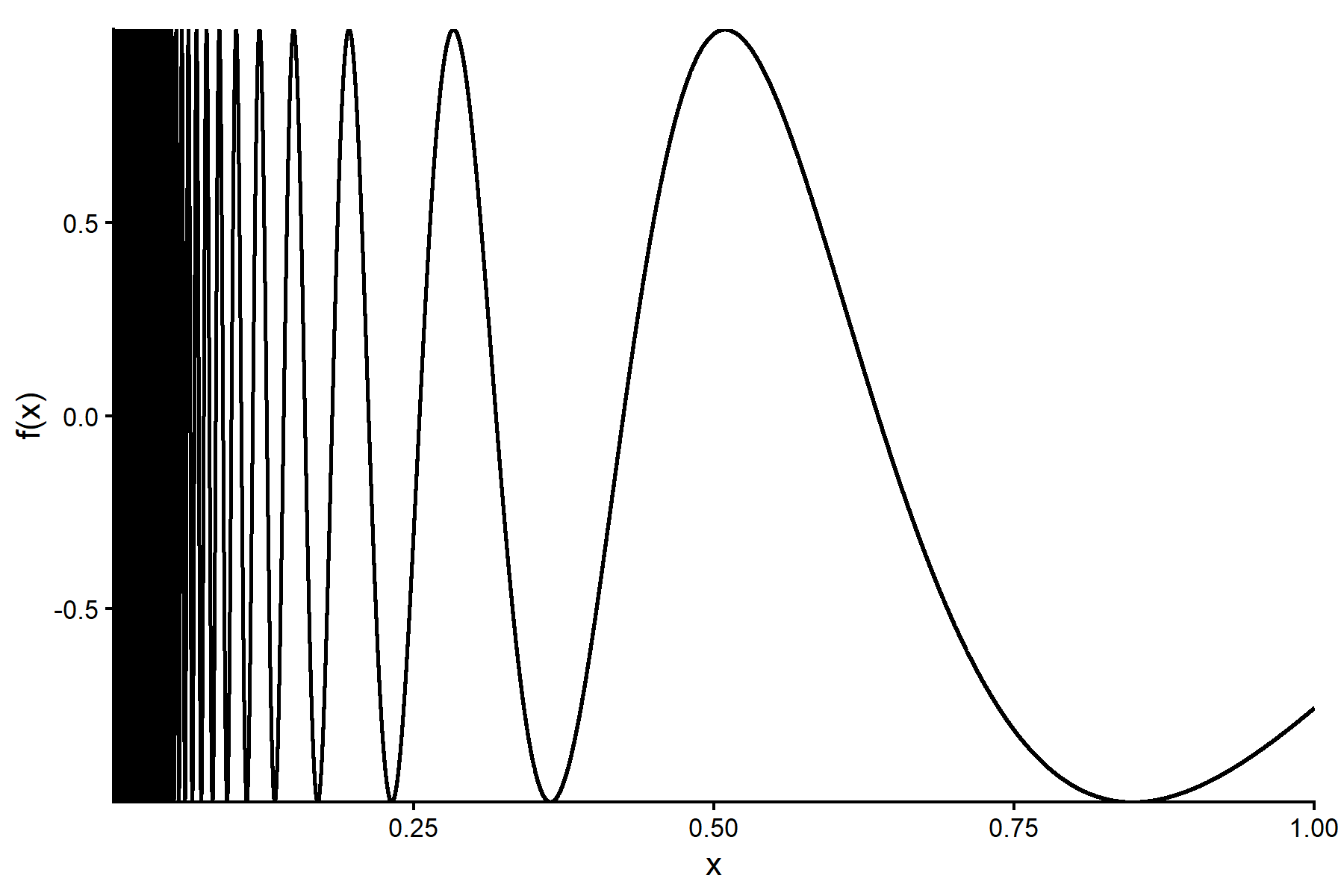}
         \caption{Signal $f_{\textsc{doppler}}$ with $\delta=1$.}
     \end{subfigure}
     \begin{subfigure}[T]{0.45\textwidth}
         \centering
         \includegraphics[width=\textwidth]{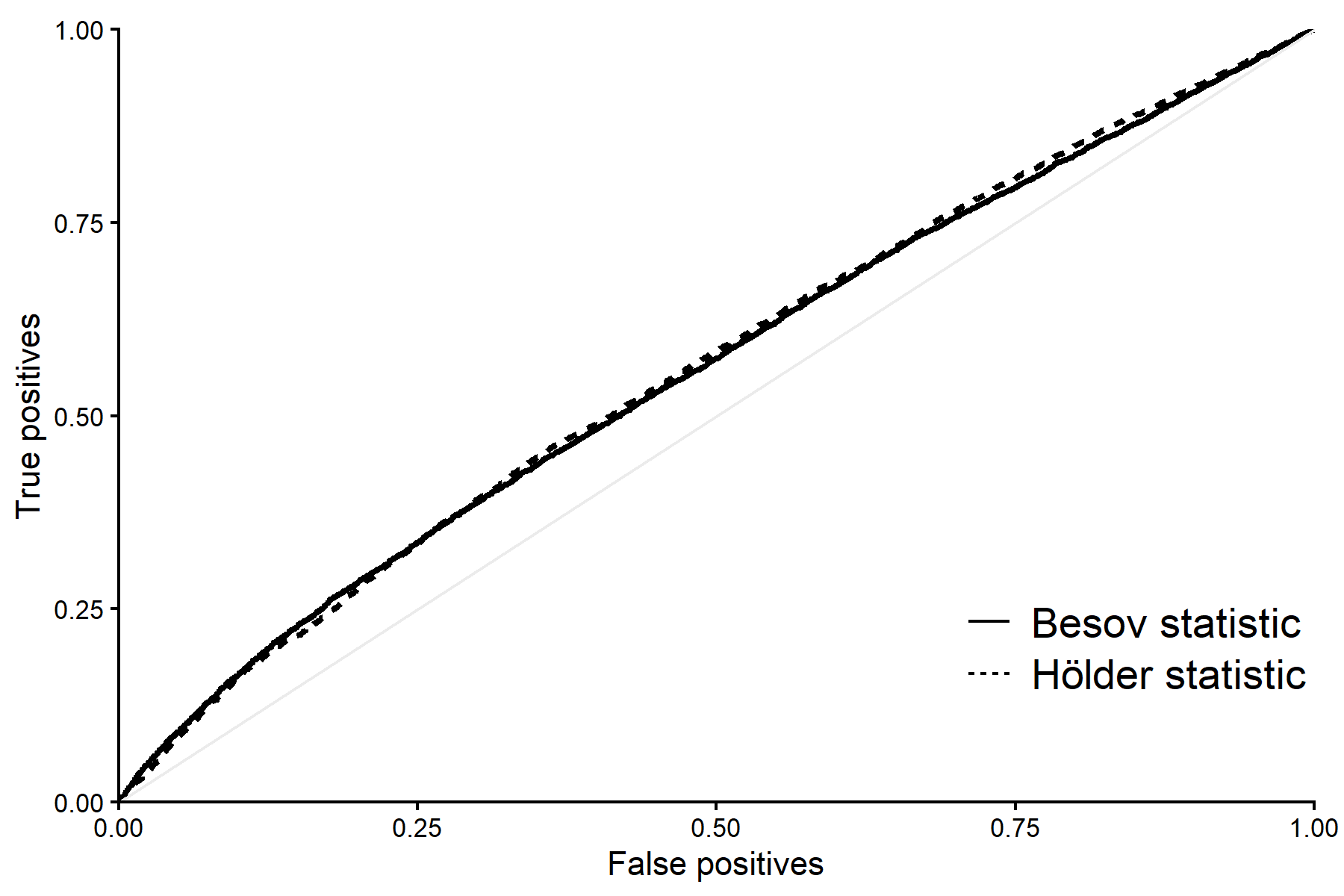}
         \caption{Power against $f_{\textsc{doppler}}$ with $\sqrt{n}\delta=2$.}
     \end{subfigure}
     \begin{subfigure}[T]{0.45\textwidth}
         \centering
         \includegraphics[width=\textwidth]{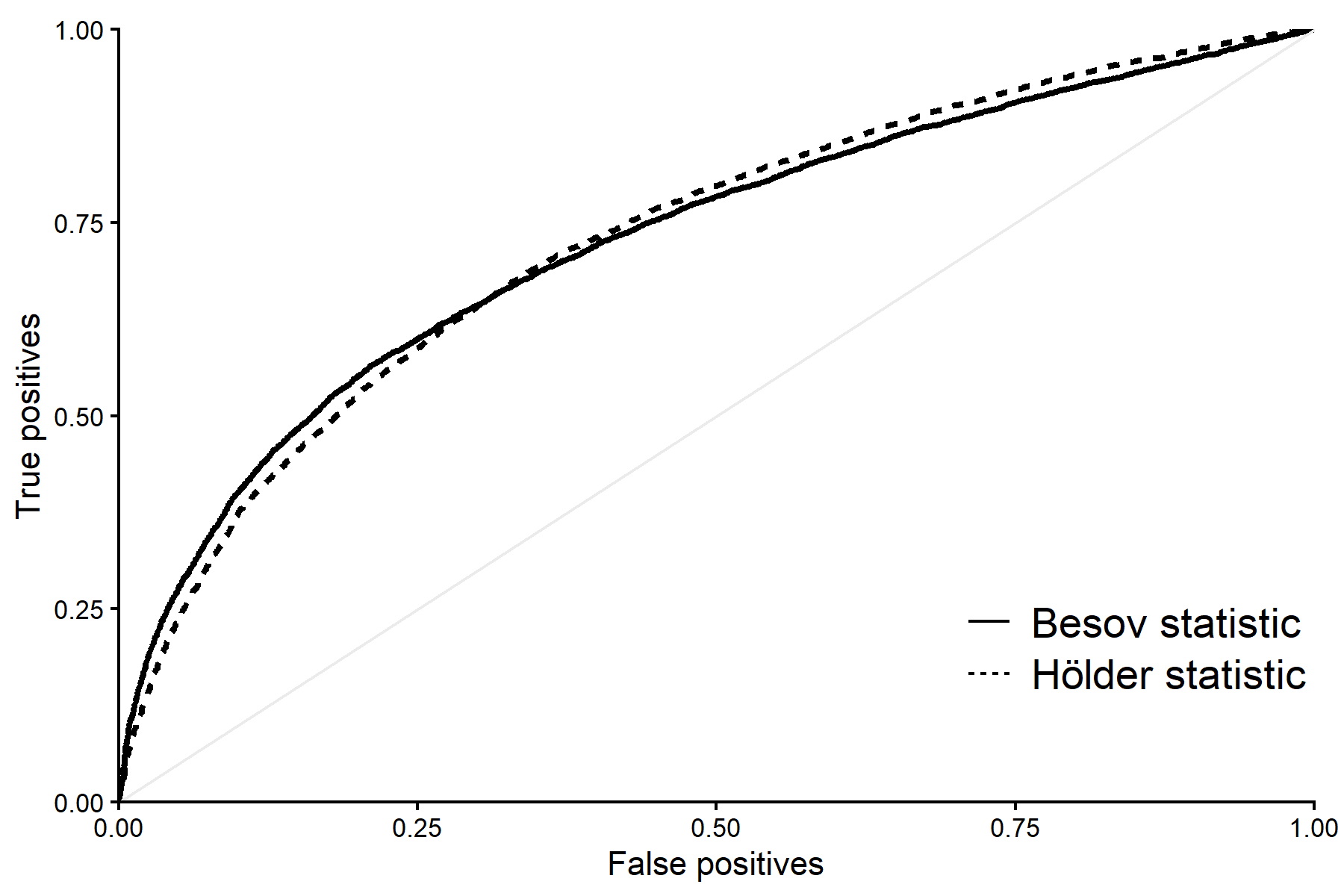}
         \caption{Power against $f_{\textsc{doppler}}$ with $\sqrt{n}\delta=4$.}
     \end{subfigure}
     \begin{subfigure}[T]{0.45\textwidth}
         \centering
         \includegraphics[width=\textwidth]{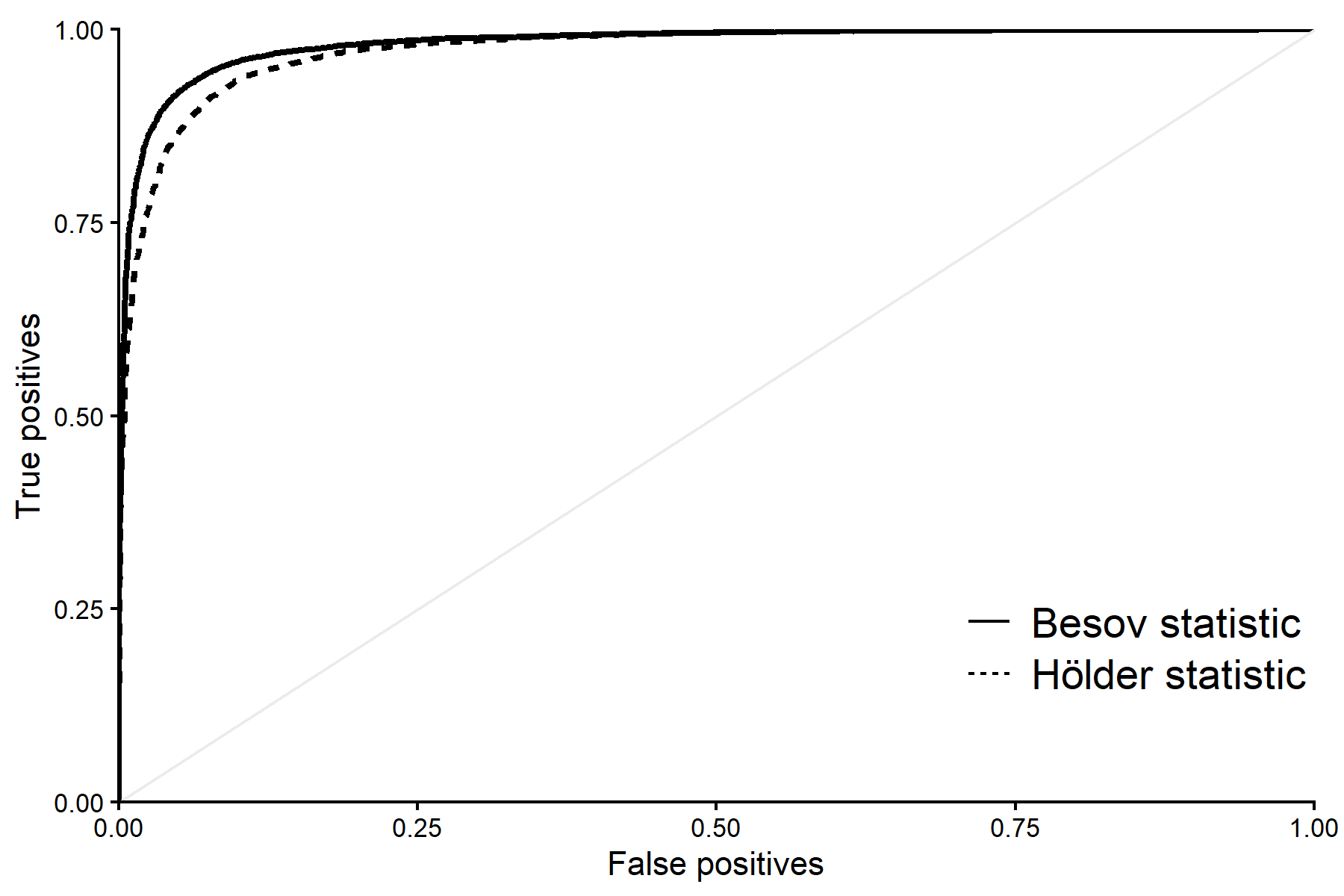}
         \caption{Power against $f_{\textsc{doppler}}$ with $\sqrt{n}\delta=8$.}
     \end{subfigure}
    \caption{Example test signal (a) and ROC curves (b-d) for the Besov-Orlicz statistic $\|Y_n\|_{\B}$ and the Hölder statistic $\|Y_n\|_{C^{\rho_{1/2}}}$ in the white noise model, with alternatives $f_{\textsc{doppler}}$. Reported quantities are based on $10^4$ Monte Carlo simulations.}
    \label{fig:doppler_power}
\end{figure}

\section{Discussion}

While the functional central limit theorem is formulated for iid random variables, an extension to independent but nonstationary data is possible because the tightness result, Theorem \ref{thm:modulus-fine}, does not require stationarity.
For example, one might study a locally stationary array $X_{t,n}\sim P_{t/n}$ of rows of independent random variables, where $P_u$ is a probability distribution on $\mathcal{X}$ for each $u\in[0,1]$ such that $\|X_{t,n}\|_{\Psi_2}$ admits an upper bound. 
Existing invariance principles for locally stationary data show that $S_n\wconv W$ in $C[0,1]$ for an inhomogeneous Brownian motion $W$, see e.g.\ \citet{mies_functional_2023}, and Theorem \ref{thm:modulus-fine} allows to conclude (thresholded) weak convergence in $B^{\rho}_{\Psi_2,\infty}$.

The independence of the random variables is only needed for condition (ii) of Theorem \ref{thm:modulus-fine}, and the partial sum process would still satisfy this condition if the $X_t$ were $m$-dependent instead of independent. 
However, a tightness result for more general temporal dependence structures would require an adaptation of the proof tailored to the specific ergodicity assumptions, such as mixing conditions \citep{bradley_basic_2005} or a decaying functional dependence measure in a nonlinear Bernoulli shift model \citep{Wu2005}.

Regarding the regularity results for the Brownian motion, it is not clear whether the rates established in Theorem \ref{thm:modulus-Brownian} are sharp, and whether the difference between the almost sure rate and the rate in probability is indeed a property of Brownian motion, or merely a limitation of our proof. 
One approach could be to derive a distributional limit of $h^{-1/4}[ \omega_{\Psi_2}(h,W)/\sqrt{h}-\sqrt{8/3} ]$ for $h\to 0$, as done by \cite{marcus_clt_2008} for the $L_p$ modulus, which provides a promising avenue for future investigations.

\appendix
\section{Proofs}\label{sec:proofs}

\subsection*{Proof of Proposition   \ref{prop:embedding}: Embeddings}
    \underline{Consider $\mathcal{X}=\R$ first.} 
    We start by showing that that the Besov-Orlicz space can be embedded in $C[0,1]$.
    Consider $\mathcal{X}=\R$, and let $r>0$ such that $\rho(h)\ll \rho_r(h):=h^{r}$ as $h\to 0$.
    Then $B^{\rho}_{\Psi_2,\infty} \subset B^{r}_{p,\infty}$ for any $p>1$, and the latter space embeds into $C[0,1]$ by standard Besov embeddings. 
    That is, for any $f:[0,1]\to \R$ such that $\|f\|_{B^{\rho}_{\Psi_2,\infty}}\leq 1$, there exists a $f^*\in C[0,1]$ with $f^*=f$ almost everywhere.
    Moreover, by the continuous embedding $B^r_{p,\infty} \hookrightarrow C^s$ for any $s<r-\frac{1}{p}$ \citep[Prop.~4.3.2]{Gine2016}, this $f^*$ is uniformly continuous with modulus of continuity $\omega(h)\leq Ch^s \|f\|_{B^{\rho}_{\Psi_2,\infty}}$, for some $C$ not depending on $f$.
    In the sequel, we identify $f$ with $f^*$.
    
    Now let $x^*\in [0,1]$ and $h\in[0,1]$ such that $|f(x^*+h)-f(x^*)|=\sup_{x\in[0,1]} \sup_{\tilde{h}\leq h} |f(x+\tilde{h})-f(x)| = \omega(f,h)$.
    Then $|f^*(x+h)-f(x)|\geq \omega(f,h)/2$ for all $|x-x^*|\leq \Delta:=(\omega(f,h)/(4C))^{1/s} $, and thus
    \begin{align*}
        &\omega_{\Psi_2}(f,h)=\|f(\cdot +h)-f(\cdot)\|_{\Psi_2} \\
        &\geq \inf\left\{K>0\,:\, \int_{x^* \pm \Delta} \Psi_2\left(\frac{|f(x+h)-f(x)|}{K}\right)\,dx \leq 1 \right\}\\
        &\geq \inf\left\{K>0\,:\, \Delta\cdot \Psi_2\left(\frac{\omega(f,h)}{2K}\right)\,dx \leq 1 \right\} \\
        & = \frac{\omega(f,h)}{2\Psi_2^{-1}(\frac{1}{\Delta})} 
        \quad = \frac{\omega(f,h)}{2 \sqrt{\log(1+\frac{1}{\Delta})}} \\
        &\geq \frac{\omega(f,h)}{C \sqrt{\log(e+\frac{1}{\omega(f,h)})}},
    \end{align*}
    for a sufficiently big universal $C$.
    Now note that the function $\varphi:[0,\omega(f,1)]\to \R, z\mapsto z/{\sqrt{\log(e+1/z)}}$ is strictly increasing, and thus 
    \begin{align*}
        \omega(f,h)\leq \varphi^{-1}(C\omega_{\Psi_2}(f,h)).
    \end{align*}
    Moreover, observe that
    \begin{align*}
        \varphi\left(z \sqrt{\log (e+1/z)}\right) = z \cdot \sqrt{\frac{\log(e+\frac{1}{z})}{\log\left(e+\frac{1}{z\sqrt{\log(e+\frac{1}{z})}}\right)}}
        \quad \geq z.
    \end{align*}
    This allows us to conclude that $\varphi^{-1}(y)\leq y \sqrt{\log(e+1/y)}$.
    Thus, for a potentially bigger $C$,
    \begin{align*}
        \omega(f,h)
        \leq \varphi^{-1}\left( C \omega_{\Psi_2}(f,h) \right)
        &\leq C \omega_{\Psi_2}(f,h) \sqrt{\log(e+1/\omega_{\Psi_2}(f,h))}\\
        &\leq C  \rho(h) \sqrt{\log(e+1/\rho(h))} \\
        &\leq C \rho(h) \sqrt{\log(e/h)} \quad = C\rho^*(h),
    \end{align*}
    because $\rho(h)\geq c h$.
    This proves the continuous embedding into $C^{\rho^*}$ for real-valued functions. 
    
    \underline{Now let $\mathcal{X}$ be a separable Banach space} with countable dense subset $\mathcal{D}$.
    For any $f\in B^{\rho}_{\Psi_2,\infty}$, and $y\in \mathcal{D}$, define the real-valued function $f_y(x)=\|x-y\|$ such that $|f_y|_{B^{\rho}_{\Psi_2,\infty}} \leq |f|_{B^{\rho}_{\Psi_2,\infty}}$ by the triangle inequality.
    As a consequence of the real-valued embedding established above, all $f_y$ admit a joint, uniform modulus of continuity $\omega(h)\leq C\rho^*(h)$. 
    Hence, there exists a $f^*_y\in C[0,1]$ and a set $\mathcal{A}_y$ with $|\mathcal{A}_y|=1$ and $f^*_y(x)=f_y(x)$ for all $x\in \mathcal{A}_y$. 
    Let $\mathcal{A}=\bigcap_{y\in \mathcal{D}}\mathcal{A}_y$ such that $|\mathcal{A}|=1$. 
    The restricted function $f:\mathcal{A}\to \mathcal{X}$ is uniformly continuous, because for any $y\in\mathcal{D}$,
    \begin{align*}
        \|f(x+h)-f(x)\| 
        &\leq \inf_{y\in \mathcal{D}} \|f(x+h)-y\| + \|f(x)-y\| \\
        &\leq \inf_{y\in \mathcal{D}} \left|\|f(x+h)-y\| - \|f(x)-y\|\right| + 2 \|f(x)-y\| \\
        &= \inf_{y\in \mathcal{D}} |f_y(x+h)-f_y(x) + 2\|f(x)-y\| \\
        &\leq \omega(h) + 2 \inf_{y\in \mathcal{D}} \|f(x)-y\| 
        \qquad = \omega(h).
    \end{align*}
    Hence, we can extend $f$ continuously from $\mathcal{A}$ to $[0,1]$ such that $\|f(x+h)-f(x)\|\leq C\rho^*(h)$. 
    This establishes the embedding into $C^{\rho^*}$ in the infinite dimensional case.

    \underline{Regarding the compactness of the embedding}, suppose that $K$ is bounded by $1$ without loss of generality. 
    Then its closure $\overline{K}$ is also bounded by one. 
    Since $K_t$ is compact and all $f\in K$ are uniformly equicontinuous as shown above, we find that $K$ is compact in $C[0,1]$ by Arzela-Ascoli.
    Thus, if $f_n\in K$ such that $f_n\to f$ in $B^{\rho'}_{\Psi_2,\infty}$, then also $f_n\to f$ uniformly, and hence $f_n(t)\to f(t)$. This yields that $\overline{K}_t=\overline{(K_t)}$, and the latter is also compact. 
    That is, the closure $\overline{K}$ satisfies the same qualitative properties as $K$, and thus to establish relative compactness, we may assume without loss of generality that $K$ is closed.

    Using again the compactness of $K$ in $C[0,1]$, let $f_n\in K$ be a sequence such that $f_n\to f$ in  uniformly, and thus also $\|f_n-f\|_{\Psi_2}\to 0$.
    We want to show that $f_n\to f$ in $B^{\rho'}_{\Psi_2,\infty}$. 
    Because $g\mapsto \omega_{\Psi_2}(g,h)$ is continuous on the domain $C[0,1]$ for each fixed $h>0$, we find that $\omega_{\Psi_2}(f,h) = \lim \omega_{\Psi_2}(f_n,h) \leq \rho(h)$.
    Hence, $\|f\|_{B^{\rho}_{\Psi_2,\infty}} \leq \sup_{g\in K}\|g\|_{B^{\rho}_{\Psi_2,\infty}} \leq 1$.
    Then, for any $\delta>0$,
    \begin{align*}
        &\limsup_{n\to\infty}\sup_{h\in(0,1)} \frac{\omega_{\Psi_2}(f_n-f, h)}{\rho'(h)} \\
        &\leq  \limsup_{n\to\infty} \sup_{h\in(\delta,1)} \frac{\omega_{\Psi_2}(f_n-f, h)}{\rho'(h)} \;+\; \limsup_{n\to\infty} \sup_{h\in(0,\delta]} \frac{\omega_{\Psi_2}(f_n-f, h)}{\rho'(h)} \\
        &\leq \limsup_{n\to\infty} \sup_{h\in(\delta,1)} \frac{\omega_{\Psi_2}(f_n-f, h)}{\rho'(h)} \;+\; \sup_{h\in(0,\delta]} \frac{2 \rho(h)}{\rho'(h)} \\
        &=\sup_{h\in(0,\delta]} \frac{2 \rho(h)}{\rho'(h)}. 
    \end{align*}
    The latter term tends to zero as $\delta\to 0$ as $\rho'\gg \rho$. 
    Since $\delta>0$ is arbitrary, we conclude that $|\sup_{h\in(0,1)} \frac{\omega(f_n-f, h)}{\rho(h)}| \to 0$ as $n\to\infty$, and thus $f_n\to f$ in $B^{\rho'}_{\Psi_2,\infty}$. 

    \underline{Returning to the special case $\mathcal{X}=\R$}, boundedness of $K$ in $B^{\rho}_{\Psi_2,\infty}$ implies that each $K_t$ is bounded and hence compact.
    Otherwise, if there exists a sequence $f_n\in K$ with $f_n(t)\to \infty$, then the equicontinuity of $K$ would imply that $\|f_n\|_{\Psi_2}\to \infty$ as well, which contradicts the boundedness of $f_n$ in $B^{\rho}_{\Psi,\infty}$.
    \qed

\subsection*{Proof of Theorem \ref{thm:subgauss-int-concentrate} and \ref{thm:modulus-fine}}
    The major part of the argument to establish Theorem \ref{thm:subgauss-int-concentrate} and \ref{thm:modulus-fine} is the same, hence we present it jointly.

    Denote $h_{i,j}=j2^{-i}$.
    Suppose we can establish that for some increasing function $H(h)$ with $H(0)>0$, and some $\zeta\in(0,1)$, 
    \begin{align}
        \omega_{\Psi_2}(h_{i,j},Z) \leq H(h_{i,j})\tau \sqrt{h_{i,j}}, \qquad i\geq |\log_2 \delta|,\quad  j=1,\ldots, 2^{\lfloor i \zeta \rfloor}.\label{eqn:chaining-raw-2}
    \end{align}
    Then for any $h\in(0,1)$, choose $m=m(h)\in \N$ such that $2^{-(m+1)(1-\zeta)} < h \leq 2^{-m(1-\zeta)}$.
    We can then find a $j(h)\leq 2^{\lfloor m\zeta\rfloor}$ such that $j(h)2^{-m}\leq h < [j(h)+1]2^{-m}$.
    Denote $[h] = j(h) 2^{-m(h)}$, so that $[h]$ is a left neighbor of $h$ on the dyadic grid $\{ k2^{-m(h)}\,:\, k=0,\ldots, 2^{m(h)} \}$.
    Thus, we may write $h$ as the series
    \begin{align*}
        h = [h] + \sum_{k=m}^\infty a_k(h) 2^{-k}, \qquad a_k(h)\in\{0,1\}.
    \end{align*}
    Hence, provided that \eqref{eqn:chaining-raw-2} holds, we find for any $h\leq \delta$ that
    \begin{align}
         \omega_{\Psi_2}(h, Z) 
        &\leq \omega_{\Psi_2}\left([h],Z\right) + \sum_{k=m(h)}^\infty \omega_{\Psi_2}\left(2^{-k},Z\right)\nonumber\\
        &\leq H([h])\tau \sqrt{[h]} + H(2^{-m})\sum_{k=m(h)}^\infty \tau  2^{-\frac{k}{2}} \nonumber\\
        &\leq H(h) \tau \sqrt{h} + \tau H(2^{-m(h)})2^{-\frac{m(h)}{2}}\frac{\sqrt{2}}{\sqrt{2}-1} \nonumber \\
        &\leq H(h) \tau \sqrt{h} + H(h^{\frac{1}{1-\zeta}}) \tau \sqrt{h}^{\frac{1}{1-\zeta}} \frac{\sqrt{2}}{\sqrt{2}-1}, \label{eqn:modulus-summed}
    \end{align}
    which implies 
    \begin{align*}
        \limsup_{h\to 0} \frac{\omega_{\Psi_2}(h,Z)}{\sqrt{h}} \leq \tau H(0).
    \end{align*}

    In view of the derivations above, it is sufficient to establish \eqref{eqn:chaining-raw-2}.
    That is, we need to control the probability
    \begin{align}
        &P\left( \omega_{\Psi_2}(h_{i,j}, Z) \geq H(h_{i,j}) \tau \sqrt{h_{i,j}} \;\text{ for some }i\geq |\log_2 \delta|,\; j\leq 2^{i\zeta} \right) \nonumber \\
        &\leq P\left( \omega_{\Psi_2}(h_{i,j}, Z) \geq H(h_{i,j}) \tau \sqrt{h_{i,j}} \;\text{ for some } |\log_2 \delta| \leq i \leq \left\lfloor\frac{|\log_2 (h_0/M)|}{1-\zeta}\right\rfloor, \; j\leq 2^{i\zeta} \right) \nonumber \\
        &\qquad + P\left(R>H(0) \tau \sqrt{M}\right) \nonumber\\
        &\leq  \sum_{i=\lceil |\log_2 \delta| \rceil }^{\left\lfloor\frac{|\log_2 (h_0/M)|}{1-\zeta}\right\rfloor} \sum_{j=1}^{2^{\lfloor i \zeta\rfloor}}  P\left( \omega_{\Psi_2}(h_{i,j}, Z) \geq H(h_{i,j}) \tau \sqrt{h_{i,j}} \right) \quad + P\left( R>H(0)\tau \sqrt{M}  \right),\label{eqn:tightB}
    \end{align}
    where we used condition (iii) in the last step, and $M\in \N$ will be specified later.
    To bound \eqref{eqn:tightB}, we proceed differently for condition (i), i.e.\ Theorem \ref{thm:subgauss-int-concentrate}, and condition (i)', i.e.\ Theorem \ref{thm:modulus-fine}.
    Specifically, we need to choose a different function $H$.

    \begin{proof}[Proof of Theorem \ref{thm:subgauss-int-concentrate}]
        Choose $H(h)=1+\epsilon$ constant, for any $\epsilon>0$.
        Via Lemma \ref{lem:modulus-concentrate}, we obtain 
        \begin{align}
         &\quad \sum_{i=\lceil |\log_2 \delta| \rceil }^{\left\lfloor\frac{|\log_2 (h_0/M)|}{1-\zeta}\right\rfloor} \sum_{j=1}^{2^{\lfloor i \zeta\rfloor}}  P\left( \omega_{\Psi_2}(h_{i,j}, Z) \geq H(h_{i,j}) \tau \sqrt{h_{i,j}} \right) \nonumber \\
         &\leq \sum_{i=\lceil |\log_2 \delta| \rceil }^{\left\lfloor\frac{|\log_2 (h_0/M)|}{1-\zeta}\right\rfloor} \sum_{j=1}^{2^{\lfloor i \zeta\rfloor}}  \frac{C}{\epsilon^2}(h_{i,j}\vee h_0)^{2\epsilon}\quad
         \leq \frac{C M^{2\epsilon}}{\epsilon^2}\sum_{i=\lceil |\log_2 \delta| \rceil }^{\infty} \sum_{j=1}^{2^{\lfloor i \zeta\rfloor}}  h_{i,j}^{2\epsilon} \nonumber\\
         &\leq \frac{C M^{2\epsilon}}{\epsilon^2}\sum_{i=\lceil |\log_2 \delta| \rceil }^{\infty} 2^{i(\zeta-2\epsilon)}.\nonumber
    \end{align}
    Upon choosing $\zeta<2\epsilon$, the latter series is summable.
    Now let $M=M(\delta)\to \infty$ sufficiently slow such that the latter term tends to zero as $\delta\to 0$, while at the same time $P\left( R>H(0) \tau \sqrt{M}  \right)\leq G(H(0) \tau \sqrt{M})\to 0$.
    We find that the probability in \eqref{eqn:tightB} vanishes as $\delta\to 0$, hence \eqref{eqn:chaining-raw-2} holds almost surely as $\delta\to 0$, and \eqref{eqn:modulus-summed} yields simultaneously for all $h\leq \delta$
    \begin{align*}
        \omega_{\Psi_2}(h,Z) 
        &\leq \tau \sqrt{h} \left(1+\epsilon\right)\left( \frac{2 \sqrt{2}}{\sqrt{2}-1}  h^{\frac{1}{2(1-\zeta)}-\frac{1}{2}}\right)
        &\leq \tau \sqrt{h} \left(1+\epsilon\right)\left( \frac{2 \sqrt{2}}{\sqrt{2}-1}  \delta^{\frac{1}{2(1-\zeta)}-\frac{1}{2}}\right).
    \end{align*}
    For $\delta$ small enough, the latter bound is smaller than $\tau \sqrt{h}(1+2\epsilon)$.
    Since $\epsilon>0$ is arbitrary, this establishes the claim of Theorem \ref{thm:subgauss-int-concentrate}.
    \end{proof}

    \begin{proof}[Proof of Theorem \ref{thm:modulus-fine}]
    Via Lemma \ref{lem:modulus-concentrate}, we obtain for any $p\in(1,\frac{4}{3})$, and for $H(h) = 1+h^s$,
    \begin{align*}
         &\quad \sum_{i=\lceil |\log_2 \delta| \rceil }^{\left\lfloor\frac{|\log_2 (h_0/M)|}{1-\zeta}\right\rfloor} \sum_{j=1}^{2^{\lfloor i \zeta\rfloor}}  P\left( \omega_{\Psi_2}(h_{i,j}, Z) \geq H(h_{i,j}) \tau \sqrt{h_{i,j}} \right) \\
         & \leq C(p,\kappa) \sum_{i=\lceil |\log_2 \delta| \rceil }^{\left\lfloor\frac{|\log_2 (h_0/M)|}{1-\zeta}\right\rfloor} \sum_{j=1}^{2^{\lfloor i \zeta\rfloor}} \left( \frac{H(h_{i,j})}{H(h_{i,j})-1} \right)^{p} (h_{i,j}\vee h_0)^{p-1} \\
         & \leq C M^{p-1} \sum_{i=\lceil |\log_2 \delta| \rceil }^{\left\lfloor\frac{|\log_2 (h_0/M)|}{1-\zeta}\right\rfloor} \sum_{j=1}^{2^{\lfloor i \zeta\rfloor}} \left( \frac{ H(h_{i,j})}{H(h_{i,j})-1} \right)^{p} h_{i,j}^{p-1} \\
         & \leq C H(\delta)^{p} M^{p-1} \sum_{i=\lceil |\log_2 \delta| \rceil }^{\infty}\sum_{j=1}^{2^{\lfloor i \zeta\rfloor}} \left( H(h_{i,j})-1 \right)^{-p} h_{i,j}^{p-1}\\
         & \leq C 2^p M^{p-1} \sum_{i=\lceil |\log_2 \delta| \rceil }^{\infty}  \sum_{j=1}^{2^{\lfloor i\zeta\rfloor}} h_{i,j}^{p(1-s)-1} 
         \quad =  2^p M^{p-1} \sum_{i=\lceil |\log_2 \delta| \rceil }^{\infty} 2^{-i(p(1-s)-1)} \sum_{j=1}^{2^{\lfloor i\zeta\rfloor}} j^{p(1-s)-1} \\
         & \leq C 2^p M^{p-1} \sum_{i=\lceil |\log_2 \delta| \rceil }^{\infty} 2^{i[1-p(1-s)(1-\zeta)]}.
    \end{align*}
    Upon choosing $p<\frac{4}{3}$ big enough, the latter series is summable for any $s<1-\frac{3}{4(1-\zeta)}$.
    Now let $M=M(\delta)\to \infty$ sufficiently slowly such that the latter term tends to zero as $\delta\to 0$, while at the same time $P\left( R>H(0) \tau \sqrt{M}  \right)\leq G(H(0) \tau \sqrt{M})\to 0$.
    We find that the probability in \eqref{eqn:tightB} vanishes as $\delta\to 0$, hence \eqref{eqn:chaining-raw-2} holds almost surely as $\delta\to 0$, and \eqref{eqn:modulus-summed} yields simultaneously for all $h\leq \delta$
    \begin{align*}
        \omega_{\Psi_2}(h,Z) 
        &\leq \tau \sqrt{h} \left(1+h^s + \frac{2 \sqrt{2}}{\sqrt{2}-1}  h^{\frac{1}{2(1-\zeta)}-\frac{1}{2}}\right).
    \end{align*}
    To balance the rates, we choose $\zeta = \frac{1}{6}$ and $s<\frac{1}{4}-\zeta = \frac{1}{10}$.
    This yields
    \begin{align*}
        \omega_{\Psi_2}(h,Z)
        \leq \tau \sqrt{h} \left(1 + \frac{4 \sqrt{2}}{\sqrt{2}-1}  h^{s}\right),
    \end{align*}
    which implies
    \begin{align*}
        \limsup_{h\to 0} h^{-s} \left[\frac{\omega_{\Psi_2}(h,Z)}{\sqrt{h}} - \tau\right] \leq \tau \frac{\sqrt{2}}{\sqrt{2}-1} \quad \text{almost surely}.
    \end{align*}
    Since $s<\frac{1}{10}$ is arbitrary, we conclude that the limit is indeed zero.
    \end{proof}

\subsection*{Proof of Lemma \ref{lem:modulus-concentrate}}

    Denote $\xi=\xi(r)=\sqrt{h}\tau(1 + r)$ for brevity, and assume $\tau=1$ without loss of generality.
    We observe that
    \begin{align*}
        &P\left(\omega_{\Psi_2}(h,Z)> \xi \right) \\
        &\leq P\left[\int_0^{1-h} \Psi_2\left( \tfrac{|Z(u+h)-Z(u)|}{\xi}\right)\, du > 1 \right]\\
        &= P\left[ \int_0^{1-h} \Psi_2\left( \tfrac{|Z_{u+h}-Z_u|}{\xi}\right) - \E\Psi_2\left( \tfrac{|Z_{u+h}-Z_u|}{\xi}\right)\, du > 1-\int_0^{1-h} \E\Psi_2\left( \tfrac{|Z_{u+h}-Z_u|}{\xi}\right) \, du \right] .
    \end{align*}
    
    Now observe that $x^2 \exp(x^2)\geq \exp(x^2)-1$ because $x^2\geq 1-\exp(-x^2)=x^2 \exp(-\tilde{x}^2)$ for some $|\tilde{x}|\leq |x|$, and thus
    \begin{align*}
        \frac{d}{dz} \E \exp\left(z^2\,|Z_{u+h}-Z_u|^2\right) 
        &= 2\E\left\{z|Z_{u+h}-Z_u|^2 \exp\left(z^2\,|Z_{u+h}-Z_u|^2\right)\right\} \\
        &= \frac{2}{z}\E\left\{z^2|Z_{u+h}-Z_u|^2 \exp\left(z^2\,|Z_{u+h}-Z_u|^2\right)\right\} \\
        &\geq \frac{2}{z} \E \left[\exp(z^2|Z_{u+h}-Z_u|^2)-1\right] 
        \;= \frac{2}{z} \E \Psi_2(z|Z_{u+h}-Z_u|).
    \end{align*}
    Hence, for some $\tilde{r}\in[0,r]$,
    \begin{align*}
        \E \exp\left( \tfrac{|Z_{u+h}-Z_u|}{\xi(r)} \right)^2 
        &= \E\exp\left( \tfrac{|Z_{u+h}-Z_u|}{\xi(0)} \right)^2 - r \frac{d}{dr} \E \exp\left( \tfrac{|Z_{u+h}-Z_u|}{\xi(r)} \right)^2 \big|_{r=\tilde{r}}  \\
        &= \E\exp\left( \tfrac{|Z_{u+h}-Z_u|}{\xi(0)} \right)^2 - r \frac{\xi'(\tilde{r})}{\xi(\tilde{r})^2}  \frac{d}{dz}\E\exp\left( z^2|Z_{u+h}-Z_u|^2\right)\Big|_{z=1/\xi(\tilde{r})} \\
        &\leq \E\exp\left( \tfrac{|Z_{u+h}-Z_u|}{\xi(0)} \right)^2 - r \frac{\xi'(\tilde{r})}{\xi(\tilde{r})^2}  2\xi(\tilde{r})\E\Psi_2\left(\frac{|Z_{u+h}-Z_u|}{\xi(\tilde{r})}\right) \\
        &\leq 2-\frac{2r}{1+r},
    \end{align*}
    and thus $\E\Psi_2\left( \tfrac{|Z_{u+h}-Z_u|}{\xi(r)}\right) \leq 1- \frac{2r}{1+r}$.
    We conclude that
    \begin{align*}
         1-\int_0^{1-h} \E\Psi_2\left( \tfrac{|Z_{u+h}-Z_u|}{\xi}\right) \, du
        &\geq h + (1-h) \frac{2r }{1+r}  \quad
        \geq \frac{r }{1+r}.
    \end{align*}
    We thus found
    \begin{align}
        P\left(\omega_{\Psi_2}(h,Z)\geq \xi(r) \right) 
        &\leq P\left[ \int_0^{1-h} \Psi_2\left( \tfrac{|Z_{u+h}-Z_u|}{\xi(r)}\right) - \E\Psi_2\left( \tfrac{|Z_{u+h}-Z_u|}{\xi(r)}\right)\, du > \frac{r }{1+r} \right]. \label{eqn:bound1}
    \end{align}

    For some $M$ such that $h>h_0/M$, i.e.\ $M>h_0/h$, we decompose
    \begin{align*}
        & \left| \int_0^{1-h} \Psi_2\left( \tfrac{|Z_{u+h}-Z_u|}{\xi}\right) - \E\Psi_2\left( \tfrac{|Z_{u+h}-Z_u|}{\xi}\right)  \, du \right| \\
        &\leq \sum_{i=0}^{M} \left| \sum_{j \text{ mod } (M+1) \,= i}\int_{(j-1)h}^{jh} \Psi_2\left( \tfrac{|Z_{u+h}-Z_u|}{\xi}\right) - \E(\ldots)  \, du  \right|. 
    \end{align*}
    The $i$-th term is the sum of $(hM)^{-1}$ independent random variables.
    Using the inequality $\|\sum_{i=1}^n X_i\|_{L_p}^p \leq 2 \sum_{i=1}^n\|X_i\|_{L_p}^p$ for any $p\in[1,2]$ and independent centered random variables $X_i$, see \cite[Thm.~2]{von_bahr_inequalities_1965}, we find
    \begin{align}
        & \left\| \int_0^{1-h} \Psi_2\left( \tfrac{|Z_{u+h}-Z_{u}|}{\xi}\right) - \E(\ldots)  \, du \right\|_{L_p}\nonumber \\
        &\leq \sum_{i=0}^{M} \left\| \sum_{j \text{ mod } (M+1) \,= i}\int_{(j-1)h}^{jh} \Psi_2\left( \tfrac{|Z_{u+h}-Z_u|}{\xi}\right) - \E(\ldots)  \, du  \right\|_{L_p} \nonumber \\
        &\leq 2^{\frac{1}{p}} M (hM)^{-\frac{1}{p}} \max_j  \int_{(j-1)h}^{jh}\left\| \Psi_2\left( \tfrac{|Z_{u+h}-Z_u|}{\xi}\right) - \E(\ldots)  \right\|_{L_p} \, du \nonumber \\
        &\leq 2 (hM)^{1-\frac{1}{p}} \max_{u\in[0,1-h]} \left\| \Psi_2\left( \tfrac{|Z_{u+h}-Z_u|}{\xi}\right) - \E(\ldots)  \right\|_{L_p} \nonumber \\
        &\leq 4 (hM)^{1-\frac{1}{p}} \max_{u\in[0,1-h]} \left[\E \Psi_2\left( \tfrac{|Z_{u+h}-Z_u|}{\xi}\right)^p\right]^{\frac{1}{p}} \nonumber \\
        &\leq 16 (h\vee h_0)^{1-\frac{1}{p}} \max_{u\in[0,1-h]} \left[\E \Psi_2\left( \tfrac{|Z_{u+h}-Z_u|}{\xi}\right)^p\right]^{\frac{1}{p}} \label{eqn:moment-tail} . 
    \end{align}
    In the last step, we specify $M=1$ if $h>h_0$, and otherwise $M=\lceil h_0/h\rceil\geq 2 h_0/h$. 
    In both cases, $hM\leq 2 (h\vee h_0)$.

    To bound \eqref{eqn:moment-tail}, we use the conditions (i) resp.\ (i)'.
    
    \underline{Under condition (i)}, Lemma \ref{lem:subgauss-shrink} below yields
    \begin{align*}
        \eqref{eqn:moment-tail}
        \leq 16 (h\vee h_0)^{1-\frac{1}{p}} \left( 2 + \frac{2}{\frac{(1+r)^2}{p}-1} \right)^{\frac{1}{p}}
    \end{align*}
    In combination with \eqref{eqn:bound1}, we obtain
    \begin{align*}
        P\left( \omega_{\Psi_2}(h,Z)>\xi(r) \right) 
        &\leq 16^2 (h\vee h_0)^{p-1} \left( 2 + \frac{2}{\frac{(1+r)^2}{p}-1} \right) \\
        &= C (h\vee h_0)^{p-1} \left( \frac{(1+r)^2}{(1+r)^2-p} \right).
    \end{align*}
    Letting $p=1+\eta r\leq 2$ for some $\eta\in(0,2)$ and $r\in[0,\frac{1}{2}]$, we find
    \begin{align*}
        P\left( \omega_{\Psi_2}(h,Z)>\xi(r) \right) 
        &= C (h\vee h_0)^{\eta r} \left( \frac{1}{r^2+(2-\eta)r} \right), \quad r\in[0,\tfrac{1}{2}].
    \end{align*}
    
    \underline{Under condition (i)'}, Lemma \ref{lem:condition-sharper} below implies that
    \begin{align*}
        \eqref{eqn:moment-tail} \quad
        \leq \quad  16 (h\vee h_0)^{1-\frac{1}{p}}\left[ 1-\frac{3}{4} \frac{p}{(1+r)^2} \right]^{-\frac{1}{2p}},
    \end{align*}
    which is finite for $p< \frac{4}{3}$
    In combination with \eqref{eqn:bound1}, we obtain
    \begin{align*}
        P\left( \omega_{\Psi_2}(h,Z)>\xi(r) \right) 
        &\leq 4^p \kappa^p \left(\frac{r}{1+r}\right)^{-p} (h\vee h_0)^{p-1} \left[ 1-\frac{3}{4} \frac{p}{(1+r)^2} \right]^{-\frac{1}{2}} \\
        &\leq 16^p (1+\kappa)^p \left(\frac{r}{1+r}\right)^{-p} (h\vee h_0)^{p-1} \left[ 1-\frac{3}{4} p \right]^{-\frac{1}{2}}.\qed
    \end{align*}

    \begin{lemma}\label{lem:condition-sharper}
    Condition (i)' implies 
    \begin{align}
        \E \exp\left( \|Z_u-Z_v\|^2/K^2 \right) \leq  \kappa \left[ 1-\frac{3}{4}\frac{\tau^2|u-v|}{K^2} \right]^{-\frac{1}{2}}. \label{eqn:exp-moment}
    \end{align}
    \end{lemma}
    \begin{proof}[Proof of Lemma \ref{lem:condition-sharper}]
        We generalize the argument presented by \cite{leskela_sharp_2026} to the extended tail condition (i)' for Banach spaces.
        Let $Z$ be a random variable taking values in $(\mathcal{X},\|\cdot\|)$ such that $\E \exp(\lambda \|Z\|)\leq \exp(C^2\lambda^2/2)$ for all $\lambda\in\R$.
        For any $K>0$, and a standard normal random variable $\eta$, independent of $Z$,
        \begin{align*}
            \exp\left( \|Z\|^2/K^2 \right) 
            = \int \exp(\eta\cdot \sqrt{2} \|Z\|/K )\, dP(\eta),
        \end{align*}
        and thus
        \begin{align*}
            \E \exp\left( \|Z\|^2/K^2 \right) 
            &= \E \exp(\eta\cdot \sqrt{2} \|Z\|/K ) \\
            &\leq \kappa \E \exp\left( C^2 \eta^2 / K^2 \right) \\
            &= \kappa \int_{-\infty}^\infty \frac{1}{\sqrt{2\pi}} e^{-x^2/2} \exp\left( C^2 x^2/K^2 \right)\, dx 
            \qquad = \frac{\kappa}{\sqrt{1-2C^2/K^2}},
        \end{align*}
        for any $K>\sqrt{2}C$.
        With $C^2=\frac{3}{8}\tau^2|u-v|$, we obtain \eqref{eqn:exp-moment}.
        In particular, for $K=C \sqrt{8/3}$, we find that (i)' implies (ii).
    \end{proof}

\begin{lemma}\label{lem:subgauss-shrink}
    For a sub-Gaussian random variable $X$ taking values in a normed space $(\mathcal{X},\|\cdot\|)$, and $1\leq p < K^2$, 
    \begin{align*}
        \E \Psi_2\left(\frac{\|X\|}{K \|X\|_{\Psi_2}}\right)^p &\leq \int_0^\infty \frac{2}{(1+v^{1/p})^{K^2}}\, dv \\
        &\leq \begin{cases}
            2 + \frac{2}{\frac{K^2}{p}-1}, & p>1, \\
            \frac{2}{K^2-1}, & p=1.
        \end{cases} 
    \end{align*}
\end{lemma}
\begin{proof}[Proof of Lemma \ref{lem:subgauss-shrink}]
    By rescaling, let $\|X\|_{\Psi_2}=1$ without loss of generality.
    Then $\E \Psi_2(\|X\|)\leq 1$, i.e.\ $\E \exp(\|X\|^2)\leq 2$, and thus 
    \begin{align*}
        \E \Psi_2(\|X\|/K)^p
        &= \int_0^\infty P\left( \exp(\|X\|^2/K^2) > 1+v^{1/p} \right) \, dv \\
        &= \int_0^\infty P\left( \exp(\|X\|^2)>(1+v^{1/p})^{K^2} \right)\, dv \\
        &\leq \int_0^\infty \frac{2}{(1+v^{1/p})^{K^2}}\, dv \\
        &\leq 2 + 2 \int_1^\infty v^{-K^2/p}\, dv = 2 + \frac{2}{\frac{K^2}{p}-1}.
    \end{align*}
    Note that this bound is very loose.
    For $p=1$, the integral can be computed explicitly by shifting $v\mapsto v+1$.
\end{proof}

\subsection*{Proof of Theorem \ref{thm:modulus-Brownian}: Regularity of Brownian motion}

First, note that $\|W(1)\|_{\Psi_2}<\infty$ by virtue of Fernique's Theorem, which states that $\E\exp(\eta |W(1)|^2)<\infty$ for some $\eta>0$; see e.g.\ \cite[Thm.~2.7]{da_prato_stochastic_2014}.
    
    The upper bound $\limsup_{h\to 0} \omega_{\Psi_2}(h,W)/\sqrt{h}\leq \|W(1)\|_{\Psi_2}$ is a consequence of Theorem \ref{thm:subgauss-int-concentrate} with $h_0=0$ and $\tau= \|W(1)\|_{\Psi_2}$, and the Brownian self-similarity $W(u+h)-W(u) \deq \sqrt{h}W(1)$.
    
    For the respective lower bound, choose any $K<\|W(1)\|_{\Psi_2}$ and write
    \begin{align*}
        \int_0^{1-\frac{1}{n}} \Psi_2\left(\frac{|W(u+\frac{1}{n})-W(u)|}{K \sqrt{\frac{1}{n}}}\right) 
        &= \sum_{i=0}^{n-1} \int_{\frac{i}{n}}^{\frac{i+1}{n}} \Psi_2\left(\frac{|W(u+\frac{1}{n})-W(u)|}{K \sqrt{\frac{1}{n}}}\right) \, du \\
        &\deq \frac{1}{n}\sum_{i=0}^{n-1} \int_{0}^{1} \Psi_2\left(\frac{|W(i+v+1)-W(i+v)|}{K}\right) \, dv,
    \end{align*}
    using the self-similarity of the Brownian motion. 
    Since the processes $[W(i+v+1)-W(i+v)]_{v\in[0,1]}$ and $[W(j+v+1)-W(j+v)]_{v\in[0,1]}$ are standard Brownian motions, and independent for $|i-j|\geq 2$, the standard law of large numbers yields
    \begin{align*}
        &\frac{1}{n}\sum_{i=0}^{n-1} \int_{0}^{1} \Psi_2\left(\frac{|W(i+v+1)-W(i+v)|}{K}\right) \, dv \\ 
        &\asconv \E \left[ \int_{0}^{1} \Psi_2\left(\frac{|W(v+1)-W(v)|}{K}\right) \, dv\right] 
        \qquad =\E \Psi_2(|W(1)|/K).
    \end{align*}
    If $K<\|W(1)\|_{\Psi_2}$, then this limit is larger than $1$. 
    Hence, for any $K<\|W(1)\|_{\Psi_2}$, almost surely there is a (random) $N$ such that $\omega_{\Psi_2}(\frac{1}{n},W) \sqrt{n}\geq K$ for all $n\geq N$. 
    Hence, almost surely, $\limsup_{h\to 0} \omega_{\Psi_2}(h,W)/\sqrt{h}\geq \|W(1)\|_{\Psi_2}$.
    \qed

\subsection*{Proof of Theorem \ref{thm:FCLT} and Theorem \ref{thm:convergence}: Functional weak convergence}

\begin{lemma}\label{lem:subgauss-Banach}
    Let $X_1,\ldots, X_n$ be independent centered random variables taking values in a type-2 Banach space $(\mathcal{X},\|\cdot\|)$.
    Then there exists a constant $K=K(\mathcal{X})$ such that
    \begin{align*}
        \left\| \sum_{i=1}^n X_i \right\|_{\Psi_2} \;\leq\; K \sqrt{ \sum_{i=1}^n \|X_i\|^2_{\Psi_2} }.
    \end{align*}
\end{lemma}
\begin{proof}[Proof of Lemma \ref{lem:subgauss-Banach}]
    For iid Rademacher random variables $\epsilon_i$, i.e.\ $P(\epsilon_i=1)=P(\epsilon_i=-1)=1/2$, we may use the symmetrization bound \cite[Lemma 6.3]{ledoux_probability_2011} and a sub-Gaussian concentration inequality \cite[Theorem 6.21]{ledoux_probability_2011} to obtain
    \begin{align*}
        \left\|\sum_{t=1}^n X_t\right\|_{\Psi_2} 
        &\leq K\left\|\sum_{t=1}^n X_t\right\|_{L_1} + K\sqrt{\sum_{t=1}^n \|X_t\|_{\Psi_2}^2}\\
        &\leq 2K\left\|\sum_{t=1}^n \epsilon_t X_t\right\|_{L_1} + K\sqrt{\sum_{t=1}^n \|X_t\|_{\Psi_2}^2}.
    \end{align*}
    Since $\mathcal{X}$ is of type 2, 
    \begin{align*}
        \left\|\sum_{t=1}^n \epsilon_t X_t\right\|_{L_1} 
        \leq \left\|\sum_{t=1}^n \epsilon_t X_t\right\|_{L_2} 
        &\leq K(\mathcal{X})\sqrt{\sum_{t=1}^n\left\| \epsilon_t X_t\right\|_{L_2}^2} 
        \leq K_{\mathcal{X}}\sqrt{\sum_{t=1}^n\left\| \epsilon_t X_t\right\|_{\Psi_2}^2}.
    \end{align*}
\end{proof}

\begin{lemma}\label{lem:Sn-conditions-ii-iii}
    Let $X_i$ be iid centered, sub-Gaussian random variables such that $\|X_i\|_{\Psi_2}<\infty$.
    Then the interpolated partial sum process $S_n$ satisfies conditions (ii) and (iii) of Theorem \ref{thm:subgauss-int-concentrate} with $h_0=\frac{2}{n}$, and $G(r)=2/\left(\frac{r^2}{32\tau^2}-1\right)$. 
\end{lemma}
\begin{proof}[Proof of Lemma \ref{lem:Sn-conditions-ii-iii}]
    Condition (ii) is obvious.    
    For condition (iii), let
    \begin{align*}
        R^* = \inf\left\{  r\,:\, \frac{1}{n} \sum_{i=1}^n \Psi_2\left(\frac{\|X_i\|}{r}\right) \leq 1\right\},
    \end{align*}
    with the convention $X_0=0$.
    For $h\leq \frac{1}{n}$, and $u\in [\frac{k}{n}, \frac{k+1}{n})$, the construction of $S_n$ yields $\|S_n(u+h)-S_n(u)\| \leq h \sqrt{n} \max(\|X_{k+2}\|,\, \|X_{k+1}\|)$.
    Thus, with the notational convention $X_{n+1}=X_n$ and $X_0=0$, and using $h_0=2/n$,
    \begin{align*}
        &\int_0^{1-h} \Psi_2\left( \frac{\|S_n(u+h)-S_n(u)\|}{\sqrt{8}\sqrt{h} \sqrt{\frac{h}{h_0}} R^*} \right)\, du \\
        &\leq \frac{1}{n} \sum_{k=0}^{n-1} \Psi_2\left( \frac{\sqrt{n} h}{\sqrt{8}\sqrt{h} \sqrt{\frac{h}{h_0}} R^*} \left(\|X_{k+2}\| \, +\, \|X_{k+1}\|\right) \right) \\
        &= \frac{1}{n} \sum_{k=0}^{n-1} \Psi_2\left( \frac{1}{2 R^*} \left(\|X_{k+2}\|\, +\, \|X_{k+1}\|\right) \right)  
        \qquad \leq 1.
    \end{align*}
    Hence, condition (iii) holds with $R=\sqrt{8} R^*$. 
    Moreover, 
    \begin{align*}
        P(R^*>r) 
        &\leq  P \left( \frac{1}{n} \sum_{i=1}^n \Psi_2\left(\frac{\|X_i\|}{r}\right) >\frac{1}{2}\right) \\
        &\leq 2\, \E \left[ \frac{1}{n} \sum_{i=1}^n \Psi_2\left(\frac{\|X_i\|}{r}\right)\right] \\
        &\leq \frac{2}{(\frac{r}{2\tau})^2-1}=:G(r) \overset{r\downarrow 0}{\longrightarrow} 0,
    \end{align*}
    by virtue of Lemma \ref{lem:subgauss-shrink}.
\end{proof}

\begin{proof}[Proof of Theorem \ref{thm:FCLT}]
    Since $\mathcal{X}$ is a Banach space of type 2, and the random variables $X_t$ have finite second moments, they satisfy the central limit theorem, see \cite{hoffmann-jorgensen_law_1976}.
    From condition (ii) of Theorem \ref{thm:FCLT} and the independence of $X_t$, we may readily conclude that $(S_n(u_i))_{i=1,\ldots, m} \wconv (W(u_i))_{i=1,\ldots, m}$ for any $u_1,\ldots, u_m\in[0,1]$.
    To establish tightness for the invariance principle, we use Theorem \ref{thm:convergence}, and conditions (ii) and (iii) therein hold by virtue of Lemma \ref{lem:Sn-conditions-ii-iii}.
    Moreover, observe that $S_n(u)-S_n(v)=\sum_{i=1}^n X_i w_{i,n}(u,v)$ for weights with $\sum_{i=1}^n w_{i,n}^2 \leq |u-v|$.
    Hence, via Lemma \ref{lem:subgauss-Banach},
    \begin{align*}
        \|S_n(u)-S_n(v)\|_{\Psi_2} \leq C \sqrt{\sum_{i=1}^n w_{i,n}(u,v)^2 \|X_i\|_{\Psi_2}^2} \leq C \tau \sqrt{|u-v|},
    \end{align*}
    for some $C=C(\mathcal{X})$.
    This establishes condition (i) of Theorem \ref{thm:convergence}.
    Hence, $\|S_n\|_{\B}$ is stochastically bounded.
    Since $\B\hookrightarrow B_{\Psi_2,\infty}^\rho$ is compact by Proposition \ref{prop:embedding}, the sequence $S_n$ is relatively compact in $B_{\Psi_2,\infty}^\rho$. 
    We obtain weak convergence by virtue of Prokhorov's Theorem.
\end{proof}

\begin{proof}[Proof of Theorem \ref{thm:convergence}]
    By Theorem \ref{thm:FCLT}, $S_n\wconv W$ in $C[0,1]$.
    For any $\delta>0$, the mapping $T_{[\delta,1)}:f\mapsto \sup_{h\in [\delta,1)} \omega_{\Psi_2}(h,f)/\sqrt{h}$ is a continuous functional on $C[0,1]$,
    and thus $T_{[\delta,1)}(S_n) \wconv T_{[\delta,1)}(W)$.
    Now define $T_{(0,\delta)}(f) = \sup_{h\in(0,\delta)} \omega_{\Psi_2}(h,f)/\sqrt{h}$, such that $
        |f|_{\B} = \max(T_{(0,\delta)}(f), T_{[\delta,1)}(f))$.
        
    In the proof Theorem \ref{thm:FCLT}, we established the sufficient conditions of Theorem \ref{thm:subgauss-int-concentrate} for some $\tau=K(\mathcal{X}) \|X\|_{\Psi_2}$. 
    Thus for any $t>\tau$, we find $\limsup_n P(T_{(0,\delta)}(S_n)>\tau+\epsilon) \to 0$ as $\delta\to 0$.
    On the other hand, $T_{[\delta,1)}(W)\to |W|_{\B}$ as $\delta\to 0$.
    Thus, for any $t>\tau$, 
    \begin{align*}
        \limsup_{n\to\infty} P\left( |S_n|_{\B} \geq t \right) 
        &\leq \limsup_{n\to\infty} P(T_{(0,\delta)}(S_n)\geq t) + \limsup_{n\to\infty}  P(T_{[\delta,1)}(S_n)\geq t) \\
        &\leq  P(T_{[\delta,1)}(W)\geq t) \overset{\delta\downarrow 0}{\longrightarrow} P(T_{(0,1)}(W)\geq t). 
    \end{align*}
    This establishes the thresholded weak convergence by the Portmanteau Lemma.

    Now consider the special case $\mathcal{X}=\R$ with the additional condition that $\E \exp(\lambda X_t)\leq \exp(\lambda^2\sigma^2/2)$. 
    It suffices to establish condition (i) of Theorem \ref{thm:convergence} with $\tau=\sigma \sqrt{8/3}$.
    With the same weights $w_{i,n}(u,v)$ as above, we find
    \begin{align*}
        \E \exp\left( \lambda [S_n(u)-S_n(v)]\right) 
        &= \E \exp\left(\lambda \sum_{i=1}^n w_{i,n}(u,v) X_i \right) 
        \;= \prod_{i=1}^n \E \exp\left(\lambda w_{i,n}(u,v)X_i\right) \\
        &\leq \exp\left( \frac{\sigma^2}{2}\sum_{i=1}^n w_{i,n}(u,v)^2\right) 
        \; \leq \exp\left(\frac{\sigma^2|u-v|}{2}\right).
    \end{align*}
    As in Remark \ref{rem:subgauss}, we conclude that $\|S_n(u)-S_n(v)\|_{\Psi_2} \leq \sigma\sqrt{8/3} \sqrt{|u-v|}$.
\end{proof}

\subsection*{Proof of Theorem \ref{eqn:CLT-discretized}: Discretized statistics}

First, note that 
\begin{align*}
    \|S_n-\widetilde{S}_n\|_{\Psi_2(dx)}
    &\leq \inf\left\{K>0\,:\,  \sum_{t=1}^n \int_{(t-1)/n}^{t/n} \Psi_2\left( \frac{\|S_n(u)-\widetilde{S}_n(u)\|}{K} \right)\, du \right\} \\
    &\leq \inf\left\{K>0\,:\, \frac{1}{n} \sum_{t=1}^n \Psi_2\left( \frac{\|X_t\|}{K\cdot \sqrt{n}} \right) \right\} \\
    &= \frac{1}{\sqrt{n}} \inf\left\{K>0\,:\, \frac{1}{n} \sum_{t=1}^n \Psi_2\left( \frac{\|X_t\|}{K} \right) \right\}.
\end{align*}
The latter infimum is an empirical Orlicz norm of iid random variables, which converges almost surely to $\|X\|_{\Psi_2(dP)}$, see \cite{mies_empirical_2025}.
Thus, 
\begin{align*}
    \limsup_{n\to\infty} \sqrt{n} \|S_n-\widetilde{S}_n\|_{\Psi_2(dx)} \leq \|X\|_{\Psi_2}\quad \text{almost surely},
\end{align*}
and hence
\begin{align}
    \limsup_{n\to\infty} \sqrt{n}\sup_{h\in(0,1)}\left|\omega_{\Psi_2}(h, S_n) -\omega_{\Psi_2}(h, \widetilde{S}_n) \right| \leq 2\|X\|_{\Psi_2}\quad \text{almost surely}. \label{eqn:interpolation-error-1}
\end{align}
As $\rho(1/n)\gg 1/\sqrt{n}$, we conclude that $D_{\rho,n}(\widetilde{S}_n) = D_{\rho,n}(S_n)+o(1)$.
Moreover, for any $\delta>0$,
\begin{align*}
    &|D_{\rho,n}(S_n)-D_{\rho}(S_n)| \\
    &\leq \sup_{h\in[\delta,1)} \left|\frac{\omega_{\Psi_2}(\frac{\lceil nh\rceil }{n}, S_n)}{\rho(\frac{\lceil nh\rceil }{n})} - \frac{\omega_{\Psi_2}(h, S_n)}{\rho(h)} \right| 
    + \sup_{h\leq \delta} \frac{\omega_{\Psi_2}(h, S_n)}{\rho(h)} \\
    &\leq  \sup_{h\in[\delta,1)} \left|\frac{\omega_{\Psi_2}(\frac{\lceil nh\rceil }{n}, S_n) - \omega_{\Psi_2}(h, S_n)}{\rho(\delta)}  \right|
    \\
    &\quad+  \sup_{h\in[\delta,1)} \frac{\omega_{\Psi_2}(h, S_n)}{\rho(h)} \left|\frac{\rho(h)-\rho(\frac{\lceil nh \rceil}{n})}{\rho(\delta)}  \right|
    + \sup_{h\leq \delta} \frac{\omega_{\Psi_2}(h, S_n)}{\rho(h)} \\
    &= \mathcal{O}_P\left(\frac{\sqrt{1/n}}{\rho(\delta)}\right) + o_P \left( \frac{1}{\rho(\delta)}\right) + \mathcal{O}_P\left(\frac{\sqrt{\delta}}{\rho(\delta)}\right) \qquad \longrightarrow 0.
\end{align*}
In the last step, we used \eqref{eqn:interpolation-error-1} for the first term, the regularity of $S_n$ for the third term, and the continuity of $\rho$ for the second term. 
Now, letting $\delta=\delta_n\to 0$ sufficiently slowly, the whole expression tends to zero.
This shows that $D_{\rho,n}(S_n)$ and $D_{\rho,n}(\widetilde{S}_n)$ have the same limit as $D_{\rho}(S_n)$, which converges to $D_\rho(W)$ by Theorem \ref{thm:FCLT}.
The same argument applies to the dyadic statistic, where $D_{\rho, \textsc{dyadic}}(S_n)$ converges as a continuous functional of $S_n$.

Now consider the critical case $\rho(h)=\sqrt{h}$. 
For any $\delta>0$, define 
\begin{align*}
    D_{\rho,n}^{(0,\delta)}(\widetilde{S_n}) 
    &= \max_{\substack{m=1,\ldots, n,\,\frac{m}{n}<\delta}} \omega_{\Psi_2}(\tfrac{m}{n}, \widetilde{S}_n)/\rho(\tfrac{m}{n}),\\
    D_{\rho,n}^{[\delta,1)}(\widetilde{S_n}) 
    &= \max_{\substack{m=1,\ldots, n,\, \frac{m}{n}\geq \delta}} \omega_{\Psi_2}(\tfrac{m}{n}, \widetilde{S}_n)/\rho(\tfrac{m}{n}) 
\end{align*}
and analogously $D_\rho^{(0,\delta)}$ and $D_{\rho}^{[\delta,1)}$.
Then $D_{\rho,n}^{[\delta,1)}(\widetilde{S}_n)\wconv D_{\rho}^{[\delta,1)}(W)$, which in turn converges to $D_\rho(W)$ as $\delta\to 0$. 
At the same time, for any $\kappa>2\|X\|_{\Psi_2}$,
\begin{align*}
    \limsup_{n\to \infty} P\left(D_{\rho,n}^{(0,\delta)}(\widetilde{S}_n)>\tau\right)
    &\leq \limsup_{n\to \infty} P\left(D_{\rho,n}^{(0,\delta)}(S_n)>\tau-\kappa\right) \\
    &\quad + \limsup_{n\to\infty} P\left(\sup_{h\in(0,1)}\left|\omega_{\Psi_2}(h,S_n)-\omega_{\Psi_2}(h, \widetilde{S}_n)\right|\big/\rho(1/n)>\kappa\right).
\end{align*}
The second term tends to zero by \eqref{eqn:interpolation-error-1}, and the first term tends to zero for any $\tau>K\|X\|_{\Psi_2}+\kappa$ as $\delta \to 0$, as shown in the proof of Theorem \ref{thm:convergence}. 
We may thus conclude the proof as in Theorem \ref{thm:convergence}.
The dyadic statistic may be handled analogously. \qed

\subsection*{Proof of Proposition \ref{prop:lb-signal-norm-2}}
	For any $K>0$, and $h\leq \min |I_k|/2$,
	\begin{align*}
		\int_{0}^{1-h} \Psi\left( \frac{|F(v+h)-F(v)|}{K \rho(h) } \right)\, dv 
		&\geq \left(\sum_{k=1}^m |I_k|/2\right) \Psi\left( \frac{\delta h}{K\rho(h)} \right) \geq mh \Psi\left( \frac{\delta h}{K\rho(h)} \right) , 
	\end{align*}
	which equals $1$ for $K=\delta h/\rho(h)/\Psi^{-1}(\frac{1}{mh})$.
	Hence,
	\begin{align*}
		|F|_{B^\rho_{\Psi,\infty}} 
		\geq \frac{\omega_{\Psi}(h, F)}{\rho(h)} 
		&\geq \frac{\delta h}{\rho(h) \Psi^{-1}(\tfrac{1}{mh})}.
	\end{align*}
	The special cases are a consequence of the identity $\Psi_p^{-1}(y) = \log(1+y)^{1/p}$, such that
	\begin{align*}
		|F|_{B_{\Psi_p,\infty}^{\rho_\nu}} 
			\geq 
		\delta \sqrt{h} \log(e/h)^{-\frac{1}{\nu}} \log(1+\tfrac{1}{mh})^{-\frac{1}{p}}  
			\geq 
		\delta \sqrt{h} |\log\tfrac{e}{h}|^{-\frac{1}{\nu}}|\log \tfrac{e}{mh}|^{-\frac{1}{p}}.
	\end{align*}
	The norms in $B^{\rho_\nu}_{\infty,\infty}$ and $B^{1/2}_{\Psi_2,\infty}$ are limiting cases as $p\to\infty$ and $\nu\to \infty$, respectively.
    \qed

\bibliography{literature.bib}
\bibliographystyle{apalike}
\end{document}